\documentclass[12pt]{amsart} 

\usepackage{amsmath, amssymb, amsfonts, amsbsy, amsthm, latexsym, stmaryrd, mathtools, fullpage, enumerate, ytableau, shuffle, multirow, mathdots, enumerate, tikz, verbatim, comment, hhline, hyperref, arydshln}
\usetikzlibrary{calc,matrix,decorations.pathreplacing}

\newtheorem{thm}{Theorem}[section]
\newtheorem{lem}[thm]{Lemma}
\newtheorem{prop}[thm]{Proposition}
\newtheorem{cor}[thm]{Corollary}
\theoremstyle{definition}
\newtheorem{ex}[thm]{Example}
\newtheorem{rem}[thm]{Remark}
\newtheorem{defn}[thm]{Definition}
\newtheorem*{defn*}{Definition}
\newtheorem*{thm*}{Theorem}
\newtheorem*{prop*}{Proposition}
\newtheorem*{cor*}{Corollary}

\newtheorem{conjecture}[thm]{Conjecture}
\newtheorem*{question*}{Question}

\newcommand{\Z}{{\mathbb Z}}
\newcommand{\F}{{\mathbb F}}

\newcommand{\N}{{\mathbb N}}

\newcommand{\V}{\textbf{v}}

\usepackage[colorinlistoftodos]{todonotes}

\renewenvironment{proof}{\noindent {\bf Proof.}}{\qed \vskip 6pt}

\author{Ben Anzis}
\address{Department of Mathematics, University of Idaho, Moscow, ID 83844}
\email{\href{mailto:anzi4123@vandals.uidaho.edu}{{\tt anzi4123@vandals.uidaho.edu}}}
\author{Shuli Chen}
\address{Department of Mathematics, Cornell University, Ithaca, NY 14853}
\email{\href{mailto:sc2586@cornell.edu}{{\tt sc2586@cornell.edu}}}
\author{Yibo Gao}
\address{Department of Mathematics, Massachusetts Institute of Technology, \mbox{Cambridge, MA 02139}}
\email{\href{mailto:gaoyibo@mit.edu}{{\tt gaoyibo@mit.edu}}}
\author{Jesse Kim}
\address{Mathematics Department, Reed College, Portland, OR 97202}
\email{\href{mailto:jekim@reed.edu}{{\tt jekim@reed.edu}}}
\author{Zhaoqi Li}
\address{Department of Mathematics, Statistics, and Computer Science, Macalester College, St Paul, MN 55105}
\email{\href{mailto:zli@macalester.edu}{{\tt zli@macalester.edu}}}
\author{Rebecca Patrias}
\address{Universit\'e du Qu\'ebec \`a Montr\'eal \\
Montr\'eal (Qu\'ebec), Canada}
\email{\href{mailto:patriasr@lacim.ca}{{\tt patriasr@lacim.ca}}}

\begin{document}
\title{Jacobi-Trudi Determinants over Finite Fields}
\date{\today}

\begin{abstract} 
In this paper, we work toward answering the following question: given a uniformly random algebra homomorphism from the ring of symmetric functions over $\Z$ to a finite field $\F_q$, what is the probability that the Schur function $s_\lambda$ maps to zero? We show that this probability is always at least $1/q$ and is asymptotically $1/q$. Moreover, we give a complete classification of all shapes that can achieve probability $1/q$. In addition, we identify certain families of shapes where the corresponding Schur functions being sent to zero are independent events, and we look into the probability that a Schur functions is mapped to nonzero values in $\F_q$.
\end{abstract}

\maketitle

\section{Introduction}
Let $\Lambda$ denote the ring of symmetric functions over $\mathbb{Z}$.
Understanding the structure of the ring of symmetric functions is the focus of a great deal of research in algebraic combinatorics. For example,
see  \cite{macdonald1998symmetric}and \cite{stanley1999enumerative}.

In this paper, we study the structure of this ring using discrete
probability. We consider maps from the ring of symmetric functions over
$\mathbb{Z}$ to a finite field $\F_q$. In particular, one
can define a homomorphism from $\Lambda$ to $\F_q$ by giving the image of 
each of the complete homogeneous symmetric functions
$\{h_1,h_2,h_3,\ldots\}$. By the Jacobi-Trudi identity, we may express any
Schur function $s_{\lambda}$ in terms of a finite number of complete homogeneous symmetric
functions. This leads us to the following question, which we make more precise in a later section.

\begin{question*}What is the probability 
that $s_\lambda$ is sent to 0 by a uniformly random algebra homomorphism from $\Lambda$ to
$\F_q$? 
\end{question*}
We denote such a probability by $P(s_\lambda\mapsto 0)$. Moreover, one can look for families of partitions $\mathcal{F}$ such that $P(s_\lambda\mapsto 0 \mid s_\mu\mapsto 0)=P(s_\lambda\mapsto 0)$ for any $\lambda,\mu\in\mathcal{F}$ and/or $P(s_\lambda\mapsto a)=1/q$ for any $a\in\F_q$. This type of result is deeply connected to $\mathcal{F}$ being an algebraic basis for $\Lambda$, as we shall see later. 

We shall see that rectangle and staircase partitions play a special role in answering our question, as is also the case in many other areas of algebraic combinatorics. For example, in $K$-theory combinatorics, rectangles are the only shapes that are always unique rectification targets of $K$-theoretic jeu de taquin \cite{gaetz2016Knuth}; in the theory of Kashiwara crystals, rectangles are the only shapes for
which the classical crystal structure can be extended to the affine one \cite[Section 4]{shimozono2005crystals}; and rectangles and staircases are exactly the shapes for which promotion has a nice order \cite[Section 4]{haiman1992dual}. It is a mystery why these two families of shapes are special in so many contexts.

An additional motivation for our work comes from the fact that each Schur
function can be considered as a determinant via Jacobi-Trudi, so we are in
reality studying certain random matrices with entries in a finite field $\F_q$: 
matrices where certain entries are forced to be equal and certain entries may
be forced to be one or zero. This work has a similar flavor to the work of Haglund \cite{haglund1998q};  Lewis, Liu, Morales, and Panova \cite{lewis2010matrices}; Klein, Morales, and Lewis \cite{klein2014counting}; and others who studied matrices of a given rank over a finite field with specified positions equal to zero. In addition to combinatorics, there are similar questions being asked in coding theory. For example, in \cite{ravagnani2016rank}, Ravagnani studies rank-metric codes---sets of matrices of a given size over a finite field $\F_q$---which have applications in information theory.

Our paper is organized as follows. In Section \ref{sub:review} we review the basic notions of symmetric functions and partitions. We include an introduction to the Jacobi-Trudi identity and an example illustrating the main question of this work. 

Section~\ref{hooks_section} deals with Schur functions that are indexed by hooks---partitions of the form $(a,1^m)$ for $a,m\geq 0$. We prove the following result. 
\begin{prop*}[Proposition~\ref{hook_probability}]
Let $\lambda$ be a hook shape. Then $P(s_\lambda\mapsto 0)=1/q$. Furthermore, 
$P(s_\lambda\mapsto a)=1/q$ for any $a\in\mathbb{F}_q$.
\end{prop*}
As a corollary, we show that the set of Schur functions indexed by certain families of hooks form an algebraic basis for the symmetric functions. 

In Section~\ref{sec:machinery}, we introduce the main tools for the remainder of the paper: Schur matrices and operators on Schur matrices. We then proceed to prove the following in Section~\ref{sec:generalresults}.

\begin{thm*}[Corollary~\ref{cor: geq 1/q}]
We have that $P(s_\lambda\mapsto0)\geq 1/q$ for all shapes $\lambda.$
\end{thm*}

\begin{thm*}[Theorem~\ref{thm:asymtotic}]
For any shape $\lambda$, $P(s_\lambda\mapsto 0)=1/q+O(1/q^2)$. 
\end{thm*}

We also show that $P(s_\lambda\mapsto 0)$ cannot in general be expressed as a rational function in $q$ and conjecture an upper bound for the probability for any $\lambda$.

Section~\ref{sec:rectangles} focuses on Schur functions indexed by rectangles and staircases. In Section~\ref{sec:classification} we prove the following result.

\begin{thm*} [Theorem~\ref{thm:1/q}]
Suppose $P(s_\lambda\mapsto 0)=1/q$. Then $\lambda$ is either a hook, a rectangle, or a staircase.
\end{thm*}

Section~\ref{sec:independence} explore independence results for rectangles and staircases. Namely, we identity families of such shapes where $P(s_\lambda\mapsto 0\mid s_\mu\mapsto 0)=P(s_\lambda\mapsto 0)$ for any $\lambda$ and $\mu$ in the family. 

In Section~\ref{sec:othervalue}, we broaden our focus and consider questions related to $P(s_\lambda\mapsto a)$ for any $a\in\F_q$, particularly when $\lambda$ is a rectangle. We end with a short discussion of some miscellaneous results for shapes not yet discussed in Section~\ref{sec:misc}.

\iffalse
In this paper, we show that the specific type of matrices we consider, which
we call ``Jacobi-Trudi matrices,'' have very nice properties in terms of the
probability that their determinants vanish. In particular, we show that
the probability of $s_\lambda$ mapping to zero is at least $1/q$ for every
$\lambda$ (Theorem~\ref{thm: general matrix, geq 1/q}) and is asymptotically
$1/q$ as $q \rightarrow \infty$ (Theorem~\ref{thm:asymtotic}). Moreover, we
give a complete classification of all shapes $\lambda$ whose corresponding
Jacobi-Trudi matrix is singular with probability $1/q$
(Theorem~\ref{thm:1/q}). In fact, they are hooks (Definition~\ref{def:hook}),
rectangles (Definition~\ref{def:rectangle}) and staircases
(Definition~\ref{def:staircase}). In addition, we look at some conditional
probability of these shapes (Section~\ref{sec:independence}) and the
probability that $s_\lambda$ maps to other values in $\F_q$
(Section~\ref{sec:othervalue}). 
\fi

\section{Symmetric functions and partitions}\label{sub:review}
We first review some basic definitions and results related to symmetric
functions and partition shapes. Readers may refer to
\cite{stanley1999enumerative} for details.  

Let $\mathcal{S}_m$ denote the set of permutations of $\{1,2,\ldots,m\}$. A (polynomial) function $f(x_1,\cdots,x_m)$ is \textit{symmetric} if for any
$\sigma \in \mathcal{S}_m$, we have $f(x_1,\cdots,x_m)
=f(x_{\sigma(1)},\cdots,x_{\sigma(m)})$. We may extend this definition and
consider symmetric functions of infinitely many variables $f(x_1,x_2,\ldots)$,
where these functions must satisfy the property that
$f(x_{\sigma(1)},x_{\sigma(2)},\ldots,x_{\sigma(n)},x_{n+1},x_{n+2},\ldots)=f(x_1,x_2,x_3,\ldots)$
for any $\sigma\in\mathcal{S}_n$ and for any $n$. 

For any positive integer $k$, the \textit{elementary symmetric function} $e_k$ is defined as
$$e_k=e_k(x_1,x_2,\ldots)= \sum_{i_1<\cdots<i_k} x_{i_1}\cdots x_{i_k},$$
and the complete homogeneous symmetric function $h_k$ is defined as 
$$h_k = h_k(x_1,x_2,\ldots)=\sum_{i_1\leq \cdots\leq i_k} x_{i_1}\cdots x_{i_k}.$$
For example, $e_2 = x_1x_2+x_1x_3+x_2x_3+\ldots$, while $h_2 =
x_1^2+x_1x_2+x_2^2+x_1x_3+x_2x_3+x_3^2\ldots$. The sets $\{e_k\}_{k>0}$ and
$\{h_k\}_{k>0}$ form algebraic bases for $\Lambda$, the ring of symmetric functions over $\Z$.

A \textit{partition} $\lambda$ of positive integer $n$ is a sequence of
integers $(\lambda_1, \lambda_2, \cdots,\lambda_k)$ such that $\lambda_1 \geq
\lambda_2 \geq \cdots \geq \lambda_k>0$ and $\sum_{i=1}^{k}\lambda_k =n$. For
each $i$, the integer $\lambda_i$ is called the $i^\text{th}$ \textit{part} of
$\lambda$. We call $n$ the {\it size} of $\lambda$ and denote it by $\vert
\lambda \vert$.  
For simplicity, we use the notation $(b^n)$ to denote the
sequence $(\underbrace{b,\cdots,b}_{\text{$n$}})$. 

There are a few families of shapes that will be important in future sections, so we give them names that are common in the literature. 

\begin{defn}
A {\it hook (shape)} is a partition of the form $\lambda = (a,1^m)$ for $a>0$ and $m\geq0$. A {\it fattened hook} is a partition of the form $\lambda = (a^n,b^m)$ for $a,b,n,m>0$. A {\it rectangle} is a partition of the form $\lambda = (b^m)$ for $b,m>0$.
A {\it staircase} is a partition of the form $\lambda = (k,k-1,\ldots,1)$.
\end{defn}

To any partition $\lambda = (\lambda_1, \lambda_2, \cdots,\lambda_k)$ of $n$,
we can associate a \textit{Young diagram}, which is a collection of left-justified
boxes with $\lambda_i$ boxes in the $i^\text{th}$ row from the top for each
$i$. By a standard abuse of notation we will call this Young diagram $\lambda$. We
obtain the \textit{transpose} of any partition $\lambda$, denoted $\lambda'$,
by writing down the number of boxes in each column from left to right. For
example, the transpose of $(4,2,1)$ is the partition $(3,2,1,1)$. 

A \textit{semistandard Young tableau} of shape $\lambda$ is a filling of the
boxes of $\lambda$ with positive integers such that the entries weakly
increase across rows and strictly increase down columns.
To each semistandard Young tableau $T$ of shape $\lambda$, we may
associate a monomial $x^T$ given by $$x^T = \prod_{i \in \mathbb{N^+}}
x_i^{m_i},$$ where $m_i$ is the number of times the integer $i$ appears as an
entry in $T$.  

To illustrate, take $\lambda=(4,2,1)$. Then a semistandard Young tableau of shape $\lambda$ is given by 
\begin{center}
$T =$ \begin{ytableau}
1 & 1 & 2 & 4 \\
4 & 6 \\
5 \\
\end{ytableau}\ ,
\end{center}
and the corresponding monomial is $x^T=x_1^2 x_2 x_4^2 x_5 x_6$.

We can now define the Schur function $s_\lambda$ by

$$s_\lambda = \sum_{T}x^T,$$
where the sum is over all semistandard Young tableaux of shape $\lambda$. It is well-known that the Schur functions are symmetric and that they form a linear basis for $\Lambda$.

We have the following Jacobi-Trudi identity, which expresses a Schur functions in terms of either elementary symmetric functions or complete homogeneous symmetric functions. It will be our main tool throughout this paper.

\begin{thm}[Jacobi-Trudi Identity] \label{thm: JT Id.}
For any partition $\lambda = (\lambda_1,\cdots,\lambda_k)$ and its transpose $\lambda'$, we have
\[s_{\lambda} = \det{(h_{{\lambda_i}-i+j})_{i,j=1}^k},\]
\[s_{\lambda'} = \det{(e_{{\lambda_i}-i+j})_{i,j=1}^k},\]
where we define $h_0=e_0 =1$ and $h_m=e_m=0$ for $m<0$.
\end{thm}

Using the Jacobi-Trudi identity, we clarify our question of
interest with an example. First notice that a uniformly random algebra homomorphism from $\Lambda$ to $\F_q$ can be defined by sending each $h_k$ to a value in $\F_q$. Also notice that each Schur function may be expressed using a finite number of distinct $h_k$.

Suppose we want to compute
$P(s_{(2,1)}\mapsto 0 )$. We write $s_{(2,1)}$ in terms of complete
homogeneous symmetric functions using Jacobi-Trudi as follows:
  \[
    s_{(2,1)} = \left\lvert 
    \begin{array}{cc}   
      h_2 & h_3 \\
      1 & h_1 \\
    \end{array}
    \right\rvert
    = h_2h_1-h_3.
  \]
We see that $s_{(2,1)}$ is sent to zero when $h_3$ is sent to the same value as $h_2h_1$. There are $q^3$ choices for where to send the complete homogeneous symmetric functions that appear in the Jacobi Trudi expansion, and $q^2$ of these choices send $s_{(2,1)}$ to zero. Thus $P(s_{(2,1)}\mapsto 0)=1/q$.
Since $s_\lambda$ can always be expressed in terms of $\{h_k\}_{k=0}^m$ for some finite $m$, the probability may be computed in the same manner for any partition $\lambda$.

The Jacobi-Trudi identity motivates the following definition, which we will use often throughout the paper.

\begin{defn}\label{def:JT matrxi}
We call a matrix a \textit{Jacobi-Trudi matrix} if it can be obtained using the Jacobi-Trudi identity for some Schur function $s_\lambda$.
\end{defn}

For example, the matrices below are Jacobi-Trudi matrices coming from $s_{(3,1)}$. 

  \[
  \left\lvert 
    \begin{array}{cc}   
      h_3 & h_4  \\
      1 & h_1 \\
    \end{array}
    \right\rvert \qquad \left\lvert 
    \begin{array}{ccc}   
      e_2 & e_3 & e_4  \\
      1 & e_1 & e_2 \\
      0 & 1 & e_1 \\
    \end{array}
    \right\rvert
  \]

In later sections, we discuss matrices whose entries are variables $x_1,x_2,\ldots$. We call such a matrix Jacobi-Trudi if either the identification $x_i=h_i$ or $x_i=e_i$ yields a Jacobi-Trudi matrix as defined above.

\section{Hooks} \label{hooks_section}
We begin with some results for hook-shaped partitions. These results have very simple proofs, which will not be the case for other families of shapes. 

\begin{prop}\label{hook_probability}
Let $\lambda$ be a hook shape. Then $P(s_\lambda\mapsto 0)=1/q$. Furthermore, 
$P(s_\lambda\mapsto a)=1/q$ for any $a\in\mathbb{F}_q$.
\end{prop}

\begin{proof}
Let $\lambda = (a,1^b)$ for non-negative integers $a$ and $b$, and let $n=a+b=|\lambda|$. Using the
Jacobi-Trudi identity, we can write
  \[
    s_{\lambda} = \left\lvert \begin{array}{cccc} 
      h_a & h_{a+1} & \cdots & h_n \\
      1 & h_1 & \cdots & h_{n-1} \\
      0 & \ddots & \ddots & \vdots \\
      0 & 0 & 1 & h_1 \\
      \end{array}\right\rvert
  \]
Using the cofactor expansion about the first row, we see that 
  \[
    s_{\lambda} = (-1)^nh_n + f(h_1,\ldots,h_{n-1}),
  \]
for some polynomial $f\in\mathbb{F}_q[h_1,h_2,\ldots,h_{n-1}]$. For any fixed assignment of $\{h_1,\ldots,h_{n-1}\}$, the values of
$s_{\lambda}$ are uniformly distributed. Hence the distribution of
$s_{\lambda}$ is itself uniform.
\end{proof}

This result shows that our question can be framed equivalently in terms of the
elementary symmetric functions: we may define a random ring homomorphism from
$\Lambda$ to $\mathbb{F}_q$ by specifying the image of $e_k$ and obtain the
same probabilities. This is stated formally in the following corollary.  

\begin{cor} \label{cor: e and h}
Let $f: \{e_n\}_{n \in \N} \to \F_q$ and $g: \{h_n\}_{n \in \N} \to
\F_q$ be arbitrary maps, and consider the unique extension of $f$ and $g$ to homomorphisms 
from the algebra of symmetric functions over $\mathbb{Z}$ to $\F_q$. Then 
  \[
    P\big(f(s_{\lambda}) = a\big) = P\big(g(s_{\lambda}) = a\big)
  \]
for all $a \in \F_q$. 
\end{cor}
%\Ben{I don't like the way this is stated (I wrote it), but can't think of a better way. If anyone has ideas, please let me know.}
%\Becky{I think it's ok! If someone thinks of something better, that's fine, but I wouldn't actively put energy into it.}

From Jacobi-Trudi, we immediately obtain the following corollary, which will
be used many times in later proofs. 
\begin{cor} \label{cor:transpose}
For all $\lambda$, we have 
  \[
    P(s_{\lambda} \mapsto 0) = P(s_{\lambda'} \mapsto 0).
  \]
\end{cor} 

In the following proposition we see that it is easy to find infinite families of hooks where the corresponding Schur functions being sent to zero are independent events. We shall see that any such family indexes a set whose corresponding Schur functions form an algebraic basis for the symmetric functions.

\begin{prop} \label{hook_independence}
Let $\Lambda := \{\lambda^{(k)}\}_{k \in \N}$ be a collection of hook shapes
such that $|\lambda^{(k)}| = k$ for all
$k$. Then $$P\left(s_{\lambda^{(i)}}\to 0 \mid s_{\lambda^{(j)}}\to 0
\right)=P(s_{\lambda^{(i)}}\to 0)$$ for all $i,j\in\mathbb{N}$.  
\end{prop}

\begin{proof}
Using elementary probability theory and Proposition \ref{hook_probability}, we have 
  \begin{align*}
    P\left(s_{\lambda^{(i)}}\to 0 \mid s_{\lambda^{(j)}}\to 0 \right) &=
    \frac{P(s_{\lambda^{(i)}}\to 0 \; \& \; s_{\lambda^{(j)}}\to
      0)}{P(s_{\lambda^{(i)}}\to 0)} \\ &= \frac{P(s_{\lambda^{(i)}}\to 0 \; \& \;
      s_{\lambda^{(j)}}\to 0)}{P(s_{\lambda^{(j)}}\to 0)} \\
    &= P\left(s_{\lambda^{(j)}}\to 0 \mid s_{\lambda^{(i)}}\to 0 \right).
  \end{align*}
Hence we may assume without loss of generality that $j < i$. 

By the proof of Proposition \ref{hook_probability}, we see that 
$s_{\lambda^{(i)}}$ is determined by the values of $\{h_1,\ldots,h_i\}$ and
uniformly distributed for every choice of values for
$\{h_1,\ldots,h_{\ell}\}$, with $\ell < i$. Thus for any choice of values for
$\{h_1,\ldots,h_j\}$ such that $s_{\lambda^{(j)}} \mapsto 0$ (or such that
$s_{\lambda^{(j)}} \mapsto a\in\F_q$ for that matter), the values of 
$s_{\lambda^{(i)}}$ are uniformly distributed. Hence the probability that
$s_{\lambda^{(i)}} \mapsto 0$ is unaffected by the value of
$s_{\lambda^{(j)}}$. 
\end{proof}

\begin{cor} 
Let $\mathcal{H}$ be a collection of hooks satisfying the condition of Proposition
\ref{hook_independence}. Then $\{s_{\lambda}\}_{\lambda \in \mathcal{H}}$ forms an
algebraic basis of $\Lambda$.  
\end{cor}

\begin{proof}
Recall that $\{h_n\}_{n \leq N}$ forms an algebraic basis of the symmetric functions of degree $\leq N$, and write $\mathcal{H} = \{\lambda_n\}_{n \in
  \N}$ with $|\lambda_n| = n$. 
We prove by induction on $N$ that $\{s_{\lambda_n}\}_{n \leq N}$ is algebraically
equivalent to $\{h_n\}_{n \leq N}$. The case $N = 1$ is trivial because $s_{\lambda_1}=s_{(1)} = h_1$.

Now suppose that $\{s_{\lambda_n}\}_{n \leq N}$ is algebraically equivalent to
$\{h_n\}_{n \leq N}$. Recall from Proposition \ref{hook_probability} that 
  \[
    s_{\lambda_{n+1}} = (-1)^{n+1}h_{n+1} + f(h_1,\ldots,h_n),
  \]
for some polynomial $f$. By hypothesis, we may write $f(h_1,\ldots,h_n) =
g(s_{\lambda_1},\ldots,s_{\lambda_n})$ for some polynomial $g$. We may
therefore write $h_{n+1}$ as a polynomial in
$\{s_{\lambda_1},\ldots,s_{\lambda_{n+1}}\}$, so $\{s_{\lambda}\}_{\lambda \in
  \mathcal{H}}$ algebraically spans $\Lambda$. Using Jacobi-Trudi, we can
express $s_{\lambda_n}$ as a polynomial in $\{h_1,\ldots,h_n\}$, so
the set $\{s_{\lambda_n}\}_{n \leq N}$ must be algebraically independent. 
\end{proof}

\section{Schur matrices and their operators}\label{sec:machinery}
\subsection{Schur Matrices}

We were able to prove results for hook shapes without any extra machinery, but this will not be the case moving forward. We next introduce terminology and operations on matrices that will be our main tools in the proofs of remaining results.

Let $A=(a_{ij})$ be a square matrix of size $n$. For each $1\leq k \leq 2n-1$, define \textit{the $k^{\text{th}}$ diagonal} to be the collection of all entries $a_{ij}$ with $i-j = n-k$. We call the $n^{\text{th}}$ diagonal \textit{the main diagonal}. Similarly, we define the \textit{the $k^{\text{th}}$ antidiagonal} of $A$ to be the collection of all entries $a_{ij}$ with $i+j=k+1$. We call the $n^{\text{th}}$ antidiagonal \textit{the main antidiagonal}.

Next, consider a polynomial of the form $g(x_1,\ldots,x_k)=x_k-f_{k-1}$, where $f_{k-1}$ is in $\F_q[x_1,x_2,\ldots,x_{k-1}]$. In this setting, we say that $k$ is the \textit{label} of $g$. We say that a nonzero constant polynomial has label $0$ and leave the label of the zero polynomial undefined. 

It is useful to generalize the type of square matrices arising from Jacobi-Trudi identities for Schur functions. In particular, we define three decreasingly general types of square matrices: the general Schur matrices, the reduced general Schur matrices, and the special Schur matrices.

\begin{defn}
An $n\times n$ matrix $M = (M_{ij})_{i,j=1}^n$ is called a \textit{general Schur matrix} of size $n$ with $m$ variables $x_1,\cdots,x_m$ if it satisfies the following conditions:
\begin{enumerate}[(a)]
\item For each $1 \leq i \leq n$, the $i^\text{th}$ row is of the form $$(\underbrace{0,\cdots,0}_\text{$d_i$}, \underbrace{M_{i(d_i+1)},\cdots,M_{in}}_\text{nonzero entries})$$ with $0 \leq d_1 \leq \cdots \leq d_n \leq n$.
\item Every nonzero entry is either a nonzero constant in $\F_q$ or a polynomial in the form $x_k - f_{k-1}$ where $k\in[m]$ and $f_{k-1}\in\F_q[x_1,\ldots,x_{k-1}]$.
\item The labels of the nonzero entries are strictly increasing across rows and strictly decreasing down columns. In particular, the label of the upper right entry is the largest.
\end{enumerate}
\end{defn}

If we strengthen the condition on the entries of $M$ by declaring that all constant entries must be zero, we obtain a reduced general Schur matrix.

\begin{defn}
Let $M$ be a general Schur matrix of size $n$ with $m$ free variables $x_1, \cdots, x_m$. It is called a \textit{reduced general Schur matrix} if it has the additional property that constant entries must be zero.
\end{defn}

\begin{ex}\label{ex:generalSchur}
The matrix below is a general Schur matrix of size $6$ with $13$ free variables and with $(d_1,d_2,d_3,d_4,d_5,d_6)=(0,0,1,1,2,3).$ This matrix is not a reduced general Schur matrix since entry $M_{42}=3$ is a nonzero constant.
\[
\begin{bmatrix}
x_4  & x_5    & x_6-x_1x_3    & x_8-x_5^2        & x_{10}        & x_{13}+4 \\
x_2  & x_3    & x_5-x_3    & x_7-x_5^2        & x_9        & x_{12}-x_{11}+3 \\
0    & x_2-x_1    & x_4    & x_5      & x_8-x_1-x_2      & x_{11} \\
0    & 3    & x_3-3x_1x_2    & x_4   & x_7      & x_{10}-x_7x_9 \\
0    & 0      & x_2-x_1    & x_3-x_2   & x_6-4x_2+4         & 2x_9 \\
0    & 0      & 0      &x_1    & x_2-2  & x_8 \\
\end{bmatrix}
\]
\end{ex}

\begin{defn}
A reduced general Schur matrix of size $n$ with $m$ free variables $x_1,\cdots, x_m$ is called a \textit{special Schur matrix} if \begin{enumerate}[(a)]
\item none of its entries is $0$;
\item none of its entries has a nonzero constant term;
\item for any $2\times2$ submatrix, the sum of labels of the two entries on the main diagonal is the same as the sum of labels of the two entries on the main antidiagonal.
\end{enumerate}
\end{defn}
\begin{ex}
Below is an example of a special Schur matrix of size $4$. If we look at the $2\times 2$ submatrix consisting of entries in the first two rows and columns, we see the sum of the labels on the main diagonal is 5+5=10 and the sum of the labels on the main antidiagonal is 6+4=10.
\[
\begin{bmatrix}
x_5    & x_6  & x_8-x_5^2    & x_9-x_3x_6    \\
x_4    & x_5-x_2  & x_7    & x_8     \\
x_2      & x_3   & x_4    & x_5    \\
x_1      & x_2   & x_3      & x_4      \\
\end{bmatrix}
\]
\end{ex}

\subsection{Operations on general Schur matrices}\label{sec:operations}

We now define two operations on general Schur matrices $\psi$ and $\varphi$ that will be useful later. Intuitively, $\psi$ can be thought of as performing row and column reductions in the flavor of Gaussian elimination, whereas $\varphi$ can be viewed as assigning values in $\F_q$ to variables one-by-one and applying $\psi$ after each step. 

\begin{defn}
Let $M$ be a general Schur matrix of size $n$ with $m$ variables. We define an operation $\psi$ that takes general Schur matrices to reduced general Schur matrices as follows.
\begin{enumerate}[(a)]
\item If $M$ has no nonzero constant entries, then $\psi(M) = M$.
\item If $M$ has $k \geq 1$ nonzero constant entries:
\begin{enumerate}
\item[(i)] From top to bottom, use each of these $k$ entries as a pivot to make all other entries in its column zero by subtracting appropriate multiples of its row from each of the rows above.
\item[(ii)] Then, use this nonzero constant to make all other entries in its row zero using column operations.
\item[(iii)] Finally, delete the rows and columns containing these nonzero constants to obtain a reduced general Schur matrix $M'=\psi(M)$.
\end{enumerate}
\end{enumerate}
\end{defn}

In this way, $\psi(M)$ is either an empty matrix or a nonempty reduced general Schur matrix of size at most $n$ with at most $m$ variables. Notice that as long as there is one row or column of $M$ that has no nonzero constant, $\psi(M)$ will be nonempty. If $\psi(M)$ is nonempty, then $\det(\psi(M))=\alpha \det(M)$ for some nonzero constant $\alpha$. Thus if we randomly assign the variables to values in $\F_q$, we have that $P(\det{M} \mapsto 0) = P(\det{\psi(M)} \mapsto 0)$.

\begin{ex}\label{ex:psi}
Below is an example where we obtain $\psi(M_1)$ starting from $M_1$.

\begin{center}
$M_1=
\begin{bmatrix}
0    & 2x_2  & x_4    & x_5    \\
0    & 1  & 4x_3    & x_4     \\   
0    & 0   & x_1    & x_3-x_2    \\
0    & 0   & 0      & x_2      \\
\end{bmatrix}
\xrightarrow{\text{steps $(i)$ and $(ii)$}} \begin{bmatrix}
0    & 0    & x_4-8x_2x_3    & x_5-2x_2x_4   \\
0    & 1    & 0    & 0 \\
0    & 0      & x_1    & x_3-x_2    \\
0    & 0      & 0      & x_2     \\
\end{bmatrix} \newline
\xrightarrow{\text{step }(iii)} \begin{bmatrix}
0         & x_4-8x_2x_3    & x_5-2x_2x_4   \\
0         & x_1    & x_3-x_2    \\
0         & 0      & x_2     \\
\end{bmatrix} = \psi(M_1)$
\end{center}

Let $M_2$ be a Jacobi-Trudi matrix corresponding to $\lambda=(4,4,2,2)$. Then we have the following $M_2$ and $\psi(M_2)$.

$$M_2=
\begin{bmatrix}
      h_4 & h_5 & h_6 & h_7 \\
      h_3 & h_4 & h_5 & h_6 \\
      1   & h_1 & h_2 & h_3 \\
      0   & 1   & h_1 & h_2 \\
\end{bmatrix}
\psi(M_2)=\begin{bmatrix}
       h_6-h_1h_5-h_2h_4+h_1^2h_4 & h_7-h_2h_5-h_3h_4+h_1h_2h_4 \\
       h_5+h_1^2h_3-h_2h_3-h_1h_4 & h_6-h_2h_4-h_3^2+h_1h_2h_3 \\
\end{bmatrix}$$
\end{ex}
\iffalse
\begin{rem}\label{rem: psi_M}
The Jacobi-Trudi matrix $M$ for any Schur function is a general Schur matrix, and thus $\psi(M)$ is a reduced general Schur matrix. Moreover, notice that by the Jacobi-Trudi identity, any 0 in $M$ must appear in a row containing 1, and this row is deleted when we apply $\psi$. Therefore, $\psi(M)$ contains no zeroes and hence satisfies property 
$(a)$ for special Schur matrices. In addition, notice that $M$ satisfies properties $(b)$ and $(c)$ for special Schur matrices and so a simple induction shows $\psi(M)$ does as well. Hence $\psi(M)$ is a special Schur matrix. 
\end{rem}
\fi

\begin{prop}\label{prop: psi_M}Let $M$ be a Jacobi-Trudi matrix. Then $\psi(M)$ is a special Schur matrix.\end{prop}
\begin{proof} The Jacobi-Trudi matrix $M$ for any Schur function is a general Schur matrix, and thus $\psi(M)$ is a reduced general Schur matrix. Moreover, notice that by the Jacobi-Trudi identity, any 0 in $M$ must appear in a row containing 1, and this row is deleted when we apply $\psi$. Therefore, $\psi(M)$ contains no zeroes and hence satisfies property 
$(a)$ for special Schur matrices. In addition, notice that $M$ satisfies properties $(b)$ and $(c)$ for special Schur matrices, and so a simple induction shows $\psi(M)$ does as well. 
\end{proof}

\begin{defn}\label{def: phi}
Let $M$ be a reduced general Schur matrix of size $n$ with $m$ variables. We recursively define an operation $\varphi$ that takes general Schur matrices with a set of assignments of $\{x_1,x_2,\ldots,x_i\}$ for $i\leq m$ to values in $\F_q$ to reduced general Schur matrices as follows.
\begin{enumerate}[(a)]
\item $\varphi(\emptyset; \text{ any assignment}) = \emptyset$, where $\emptyset$ denotes the empty matrix.
\item $\varphi(M;x_1=a_1) = \psi(M(x_1=a_1))$, where $M(x_1=a)$ denotes the matrix obtained from $M$ by assigning value $a_1$ to $x_1$.
\item $\varphi(M;x_1=a_1,\cdots,x_i=a_i) = \varphi(\varphi(M;x_1=a_1,\cdots,x_{i-1}=a_{i-1});x_i=a_i)$ for $i\geq2$.
\end{enumerate}
\end{defn}
In this way, $\varphi(M; x_1 = a_1,\cdots, x_i=a_i)$ is either empty or a reduced general Schur matrix.
Notice that if $M' =\varphi(M; x_1 = a_1,\cdots, x_i=a_i)$ is empty, then $$P(\det{M}\mapsto0 \mid x_1=a_1,\cdots,x_{i}=a_{i}) = 0.$$ If $M'$ is nonempty, then we have $$P(\det{M}\mapsto0 \mid x_1=a_1,\cdots,x_{i}=a_{i}) = P(\det{M'}\mapsto0).$$ Also notice that as long as $M$ has one row or column where all entries have labels strictly larger than $i$, $M'$ will be nonempty, which can be easily shown by induction. These two observations will be useful in many proofs.

\section{General Results}\label{sec:generalresults}

In this section, we will prove a sharp lower bound for $P(s_\lambda\to 0)$, examine the behavior of the probability as $q$ approaches infinity, show that the probability is not always a rational function of $q$, and conjecture an upper bound for the probability. In contrast with Section~\ref{hooks_section}, the results on this section hold for any partition $\lambda$.

\subsection{Lower bound on the probability}

We can now use the tools from the previous section to prove the following theorem.

\begin{thm}
\label{thm: general matrix, geq 1/q}
Let $M$ be a reduced general Schur matrix of size $n>0$ with $m$ variables $x_1, \cdots, x_m$.
Then assigning the variables to values in $\F_q$ uniformly at random, we have that $$P(\det{M} \mapsto 0) \geq 1/q.$$
\end{thm}

\begin{proof}
We proceed by induction on the number of variables $m$. If $m=0$, matrix $M$ is the zero matrix, and the conclusion trivially holds. If $m=1$, then because the labels are strictly increasing across rows and strictly decreasing down columns, we know all entries except $M_{1n}$ are $0$. Hence $P(\det{M} \mapsto 0)=1/q$ if $n=1$ and $P(\det{M} \mapsto 0)=1$ if $n\geq 2$. In either case, $P(\det{M} \mapsto 0) \geq 1/q$.  

Suppose the assertion holds for any $m$ with $m\geq 2$ and $m <k$, and suppose $M$ is a reduced general Schur matrix with $k$ free variables. Consider the smallest label among the entries. Let it be some $i \geq 1$. If there are $n$ entries with label $i$, then by definition of reduced general Schur matrix we know $M$ is an upper triangular matrix and all the entries on the main diagonal have label $i$. In this case, for any assignment of $\{x_1,\ldots,x_{i-1},x_{i+1},\ldots,x_k\}$, we have at least one way to assign $x_i$ so that at least one of the diagonal entries equals $0$, which makes the determinant equal $0$. Thus we have $P(\det{M} \mapsto 0) \geq 1/q$. 

If instead there are at most $(n-1)$ entries with label $i$, then for any list of constants $a_1,\ldots,a_{i}$, $M'=\varphi(M; x_1=a_1,\cdots,x_{i}=a_{i})$ is a reduced general Schur matrix of size at least $1$ with at most $k-i$ variables. By the induction hypothesis, we have $P(\det{M'}\mapsto0)\geq 1/q$. Since $P(\det{M'}\mapsto0) = P(\det{M}\mapsto0 \mid x_1=a_1,\cdots,x_{i}=a_{i})$, combining all the conditional probabilities from the different assignments gives $P(\det{M} \mapsto 0) \geq 1/q$. 
\end{proof}

\begin{cor}\label{cor: geq 1/q}
We have that $P(s_\lambda\mapsto0)\geq 1/q$ for all shapes $\lambda.$
\end{cor}

\begin{proof}
Let $M$ be the Jacobi-Trudi matrix for $s_\lambda$, and note that it is a general Schur matrix. Since $M$ has at most $(n-1)$ 1's, $\psi(M)$ is a nonempty reduced general Schur matrix. Applying the previous theorem, $P(\det{M} \mapsto 0)=P(\det{\psi(M)} \mapsto 0) \geq 1/q$.
\end{proof}

In later sections, we will investigate which shapes have the property that $P(s_\lambda\mapsto 0)=1/q$ and will give a complete characterization of such shapes. Recall that we have already shown that this holds for hook shapes.

\subsection{Asymptotic Bound}

In this section, we show that $P(\det{M}\mapsto0)$ approaches $1/q$ as $q$ approaches infinity.
To do that, we first find an upper bound for the probability that a reduced general Schur matrix is singular.

\begin{lem}
For a reduced general Schur matrix $M$ of size $n$ such that every $0$ is strictly below the main diagonal, we have $P(\det{M}\mapsto0)\leq n/q$.
\end{lem}

\begin{proof}
The conclusion trivially holds when $n \geq q$, so we can assume $q>n$. We proceed by induction on $n$. When $n=1$, $M$ consists of a single entry of the form $x_j-f_{j-1}$ for some positive integer $j$, and $$P(\det{M}\mapsto0) = P(x_j=f_{j-1}) = 1/q= n/q.$$

Suppose the lemma holds for all $n<k$ for some $k\in\mathbb{N},k<q$. Let $M$ be a reduced general Schur matrix of size $k$ such that every 0 is strictly below the main diagonal, and let $i$ denote the smallest label among all the entries on the main diagonal. It suffices to show each of the conditional probabilities $P(\det M \mapsto 0 \mid x_1=a_1,x_2=a_2,\ldots, x_{i-1}=a_{i-1})\leq k/q$, where $a_1,\ldots,a_{i-1}\in \F_q$.

Assign $x_1,\cdots,x_{i-1}$ to some values $a_1,\cdots,a_{i-1}\in\F_q$, and consider $M'=\varphi(M;\ x_1=a_1,\cdots,x_{i-1}=a_{i-1})$ with size $\ell$. We know $\ell\geq 1$ since all the entries in the last column of $M$ have label at least $i$, and we know that $0$'s are strictly below the main diagonal in $M'$. If $\ell<k$, we know that $P(\det{M'}\mapsto0) \leq \ell/q \leq k/q$ by inductive hypothesis.

Hence we only need to consider the case when $\ell=k$. This means we deleted
no rows or columns when applying $\varphi$. Let the number of diagonal entries
with label $i$ be $d$. We may list the distinct entries with label $i$ that
appear on the main diagonal as $\{x_i-a_1,x_i-a_2,\ldots,x_i-a_r\}$, 
where the $a_j$ are distinct elements of $\F_q$ and $r\leq d$. Note that these entries have the form $x_i$ minus a constant because we have already assigned values in $\F_q$ to the variables $x_1,\ldots,x_{i-1}$.

If we assign $x_i$ to be something other than $a_1,a_2,\ldots,a_r$, then we
will have $d$ nonzero constants on the main diagonal. Applying $\psi$ we
obtain a reduced general Schur matrix of size at most $k-d$, and the
conditional probability in this sub-case is at most $(k-d)/q$ by our inductive
hypothesis.  
Thus
\begin{align*}
P(\det{M}\mapsto0 \mid x_1=a_1,\cdots,x_{i-1}=a_{i-1}) &=P(\det{M'}\mapsto0) \\
& \leq \frac{r}{q} \cdot 1 +\frac{q-r}{q}\cdot\frac{k-d}{q} \\
& = \frac{k}{q} - \frac{d-r}{q}- \frac{r(k-d)}{q^2} \\
&\leq \frac{k}{q}.
\end{align*}
\end{proof}

\begin{lem}\label{lem:asymtotic2}
Let $M$ be a reduced general Schur matrix of size $n\geq2$ such that every $0$ is strictly below the $(n-1)^\text{th}$ diagonal. Let $N$ be the $(n-1)\times(n-1)$ submatrix in its lower left corner. Then $P(\det{M}\mapsto0\ \&\ \det{N}\mapsto0)\leq n(n-1)/q^2$.
\end{lem}
\begin{proof}
Note that we may assume that $q>n-1$ because otherwise the result follows trivially. We again proceed by induction on $n$. Suppose $n=2$, and write $M=(M_{ij})_{i,j=1}^2$. Then $P(\det{N}=0)=1/q$ since $\det{N}=M_{21}$. Given $\det{N}=0$, we have $\det{M}=M_{11}M_{22}$. Since $P(M_{11}=0)=P(M_{22}=0)=1/q$, $P(\det{M}=0\ |\ \det(N)=0)\leq 2/q$. Hence
\[
P(\det{M}\mapsto0\ \&\ \det{N}\mapsto 0) \leq 1/q \cdot 2/q = n(n-1)/q^2.
\]

Suppose the lemma holds for all $2 \leq n < k$ for some $k\in\mathbb{N}$, and suppose $M$ is a reduced general Schur matrix of size $k$ with every 0 strictly below the $(k-1)^{th}$ diagonal. Denote the smallest label of the entries on the $(k-1)^\text{th}$ diagonal by $i$. It suffices to show that $P(\det{M}\mapsto 0\ \&\ \det{N}\mapsto0 \mid x_1=a_1,\cdots,x_{i-1}=a_{i-1})\leq k(k-1)/q$ for any choice of $a_1,\ldots,a_{i-1}\in\F_q$.  

Fix such a choice of $a_1,\ldots,a_{i-1}\in\F_q$, and consider the square matrix $M'=\varphi(M; x_1=a_1,\cdots,x_{i-1}=a_{i-1})$ of size $\ell$. Notice that none of the entries on the $(n-1)^{th}$ diagonal and the $n^{th}$ diagonal of $M$ become constant after we make the assignment $x_1=a_1,\cdots,x_{i-1}=a_{i-1}$. It follows that the first two rows and the last two columns are not deleted by $\varphi$.
Therefore, $\ell\geq 2$. Moreover, the $0$'s are still strictly below the $(\ell-1)^\text{th}$ diagonal in $M'$. %To see this, notice that in each row and column operation on an entry below the $(s-1)^{th}$ diagonal, where $s$ is the current matrix size, all the zeros on the lower left of this entry will stay below the $(s-2)^{th}$ diagonal and all the other zeros will have their diagonal numbers (the numbering of the diagonal an entry is on) decreased by 1 and thus also stay below the $(s-2)^{th}$ diagonal.

Since the determinants of $M'$ and its $(\ell-1)\times(\ell-1)$ submatrix at
the lower left corner $N'$ are nonzero constant multiples the
determinants of $M$ and $N$, respectively, $$P(\det{M}\mapsto
0 \; \& \; \det{N} \mapsto 0 \mid x_1=a_1,\cdots,x_{i-1}=a_{i-1}) =
P(\det{M'}\mapsto 0 \; \& \; \det{N'}\mapsto 0).$$ If $\ell<k$,
then $$P(\det{M}\mapsto 0\ \&\ \det{N}\mapsto 0 \mid
x_1=a_1,\cdots,x_{i-1}=a_{i-1})\leq\ell(\ell-1)/q^2 \leq k(k-1)/q^2$$ by
inductive hypothesis. 
%\Ben{There are some strange spaces showing up in the PDF here. I can't figure out why .. Probably a LaTeX auto-formatting anomaly.}
%\Becky{That's annoying, but we can leave it to the journal editor to fix (or ignore) before publication.}

We are only left with the case when $\ell=k$. Let the number of entries on the $(k-1)^\text{th}$ diagonal with label $i$ be $d\leq k-1$. We may list these entries as $\{x_i-a_1,x_i-a_2,\ldots,x_i-a_r\}$, where the $a_j\in\F_q$ are distinct and $r\leq d\leq k-1$. If we assign $x_i$ to be one of $a_1,\ldots,a_r$, the previous lemma gives that the corresponding conditional probability that $\det{M}=0$ is at most $k/q$. If we instead assign $x_i$ to be none of the $a_1,\ldots,a_r$, then we have $d$ nonzero constants on the $(k-1)^{\text{th}}$ diagonal of $M'$. Thus $\psi(M')$ is a reduced general Schur matrix of size at most $k-d$, and the corresponding conditional probability is at most $(k-d)(k-d-1)/q^2$ by inductive hypothesis.

Combining the two sub-cases, we see that

\begin{align*}
P(\det{M}\mapsto 0\ \&\ \det{N}\mapsto0 &\mid x_1=a_1,\cdots,x_{i-1}=a_{i-1})\\
&\leq  \frac{r}{q}\cdot\frac{k}{q} + \frac{q-r}{q}\cdot \frac{(k-d)(k-d-1)}{q^2}\\
&= \frac{k(k-1)}{q^2} - \frac{k(d-r)}{q^2} - \frac{d(k-1-d)}{q^2} - \frac{r[d^2-(2k+1)d+k(k-1)]}{q^3} \\
&\leq \frac{k(k-1)}{q^2}  - \frac{r[(k-1)^2-(2k-1)(k-1)+k(k-1)]}{q^3} \\ 
&\leq \frac{k(k-1)}{q^2}.
\end{align*}
\end{proof}

For the next theorem, recall that we say that $f(x)=O\left(g(x)\right)$ if there exist constants $c,n_0>0$ such that $|f(n)|\leq c \cdot |g(n)|$ for $n>n_0$. 

\begin{thm}\label{thm:asymtotic}
For any shape $\lambda$, we have $P(s_\lambda\mapsto 0)=1/q+O(1/q^2)$. 
\end{thm}

\begin{proof}
We show equivalently that $q \cdot P(s_{\lambda} \mapsto 0) \to 1$ as $q \to \infty$.

Let $M$ be the Jacobi-Trudi matrix for $s_\lambda$ and let $M'=\psi(M)$. Recall that $P(s_\lambda\mapsto 0) = P(\det{M'}\mapsto0)$. If $M'$ has size $1$, then $P(s_\lambda\mapsto 0) =1/q$ and we are done.

If $M'$ has size $n\geq 2$, denote the label of the upper right entry of $M'$ by $k$ and the $(n-1)\times(n-1)$ submatrix in its lower left corner by $N'$. Suppose $\det{N'} \neq0$. Expansion across the first row gives $\det{M'} = (-1)^{n+1}\det{N'}h_k + P(h_1,\cdots,h_{k-1})$ where $P(h_1,\cdots,h_{k-1})\in\F_q[h_1,\ldots,h_{k-1}]$. Thus for any assignment of $h_1,\cdots,h_{k-1}$, we have precisely one way to assign $h_k$ to achieve $\det{M'}=0$. Hence $P(\det{M'}\mapsto0 \mid \det{N'} \not\mapsto 0) =1/q$. If $\det{N'}=0$, by Lemma~\ref{lem:asymtotic2}, $P(\det{M'}\mapsto0\ \&\ \det{N'}\mapsto0)\leq n(n-1)/q^2$. 

Combining two cases, we have 
\begin{align*}
P(\det{M'}\mapsto0) =& P(\det{M'}\mapsto0 \mid \det{N'}\mapsto0)P(\det{N'}\mapsto0) \\
                           & + P(\det{M'}\mapsto0 \mid \det{N'}\not\mapsto 0)P(\det{N'}\not\mapsto 0) \\
                         =& P(\det{M'}\mapsto0\ \&\ \det{N'}\mapsto0) +  1/q \cdot P(\det{N'}\not\mapsto 0) \\
                         &\leq n(n-1)/q^2 + 1/q\cdot 1\\
                         =& 1/q + n(n-1)/q^2
\end{align*}
where the inequality follows from the previous lemma. 

Since by Theorem \ref{cor: geq 1/q}, $P(s_\lambda\mapsto 0) \geq 1/q$, we have that
$$ 1 = q\cdot 1/q \leq q\cdot P(s_\lambda\mapsto 0) \leq q\cdot [1/q + n(n-1)/q^2] = 1 +n(n-1)/q.$$
Taking $q\to \infty$ gives the desired result.
\end{proof}

\subsection{General Form of the Probability}\

For any partition shape $\lambda$, the probability $P(s_\lambda\mapsto0)$ is given by the number of homomorphisms that map $s_{\lambda}$ to 0 over the total number of homomorphisms, where the total number of homomorphisms is always a power of $q$. Given the results for hook shapes, one might hope that the probability is always a rational function of $q$, but this is not the case. We next give some examples that show another form the probability can take. 

\begin{prop} \label{prop: quasi-polynomial}
For $\lambda=(4,4,2,2)$, we have
\[
  P(s_{\lambda}\mapsto0)=\begin{cases}
        \frac{q^4+(q-1)(q^2-q)}{q^5} & \text{if \ $q \equiv 0 \mod{2}$}\\
       \frac{q^4+(q-1)(q^2-q+1)}{q^5} & \text{if \ $q \equiv 1 \mod{2}$.}
            \end{cases}
\]
\end{prop}
\begin{proof}
By Jacobi-Trudi, we have 
$$s_{\lambda} = \begin{vmatrix}
      h_4 & h_5 & h_6 & h_7 \\
      h_3 & h_4 & h_5 & h_6 \\
      1   & h_1 & h_2 & h_3 \\
      0   & 1   & h_1 & h_2 \\
\end{vmatrix}.$$

Denote this Jacobi-Trudi matrix by $M$. We count the number of assignments of the variables such that $M$ is singular. 

Applying $\psi$ to $M$ gives
\[
  \psi(M) =\begin{bmatrix}
       h_6-h_1h_5-h_2h_4+h_1^2h_4 & h_7-h_2h_5-h_3h_4+h_1h_2h_4 \\
       h_5+h_1^2h_3-h_2h_3-h_1h_4 & h_6-h_2h_4-h_3^2+h_1h_2h_3 \\
\end{bmatrix}
%=\begin{bmatrix}
%       h_6-a & h_7-b \\
%       h_5-c & h_6-d \\
%\end{bmatrix} \\
=\begin{bmatrix}
       h_6-f_5 & h_7-f_6 \\
       h_5-f_4 & h_6-g_5 \\
\end{bmatrix}
\]
where $f_5\in\F_q[h_1,\ldots,h_5]$, $f_6\in\F_q[h_1,\ldots,h_6]$, $f_4\in\F_q[h_1,\ldots,h_4]$, and $g_5\in\F_q[h_1,\ldots,h_5]$.

The question now turns into counting the number of assignments that makes $\psi(M)$ singular. We have $\det{\psi(M)}=(h_6-f_5)(h_6-g_5)-(h_7-f_6)(h_5-f_4).$ If $h_5\neq f_4$, then given any assignment of $h_1,\ldots, h_6$ there is exactly one choice of $h_7$ that makes the matrix singular. This gives $q^4\cdot (q-1)\cdot q\cdot 1 = q^6-q^5$ singular matrices in this case. 

If instead $h_5=f_4$, then $\det{\psi(M)}=0$ implies either $h_6 = f_5$ or $h_6=g_5$. The equality $h_5=f_4$ gives $f_5-g_5 =2h_1h_2h_3 - h_1^3h_3-h_3^2$. Depending on whether $\F_q$ has characteristic $2$, we have two cases.

If $\F_q$ has characteristic 2, then $f_5-g_5=h_3(h_3+h_1^3)$. If $f_5=g_5$, then given any assignment of $h_1,h_2,h_3,h_4,h_5,h_7$, we have one choice to make the matrix singular: $h_6=f_5$. In this subcase, there are $q+(q-1)=2q-1$ choices for $h_1$ and $h_3$. Hence there are $q\cdot(2q-1)\cdot q\cdot 1 \cdot q \cdot 1 = 2q^4-q^3$ singular matrices.

If $f_5\neq g_5$, then there are two choices for $h_6$, and we have $q^2-(2q-1)=q^2-2q+1$ choices for $h_1$ and $h_3$. This gives $q\cdot (q^2-2q+1)\cdot q\cdot 1 \cdot q \cdot 2= 2q^5-4q^4+2q^3$ singular matrices of this form.

Therefore if $\F_q$ has characteristic 2, there are $2q^5-2q^4+q^3$ singular matrices. 

Now suppose $\F_q$ does not have characteristic 2. Then $f_5-g_5 =-h_3[h_3-h_1(2h_2-h_1^2)]$. 

If $f_5-g_5=0$, then we have exactly one choice of $h_6$ to make the matrix singular and $2q^2-2q+1$ choices of $h_1$, $h_2$, and $h_3$ to make $f_5-g_5=0$. Hence we have $(2q^2-2q+1)\cdot q\cdot 1 \cdot q \cdot 1= 2q^4-2q^3+q^2$ singular matrices.

If $f_5-g_5 \neq 0$, then we have two choices of $h_6$ to make the matrix singular. In this subcase we have a total of $q^3-(2q^2-2q+1)=q^3-2q^2+2q-1$ choices of $h_1,h_2,h_3$ to make $f_5-g_5\neq 0$, which gives $(q^3-2q^2+2q-1)\cdot q\cdot 1 \cdot q \cdot 2= 2q^5-4q^4+4q^3-2q^2$ singular matrices.

Adding these two gives $2q^5-2q^4+2q^3-q^2$ singular matrices in this case.

Therefore, if $\F_q$ has characteristic $2$, we have a total of $q^6-q^5+2q^5-2q^4+q^3$ singular matrices; if $\F_q$ does not have characteristic $2$, then we instead have $q^6-q^5+2q^5-2q^4+2q^3-q^2$ of them.  Dividing each of these two polynomials by $q^7$ gives us the desired probabilities in the two cases.
\end{proof}

\begin{rem}
We also found that for $\lambda=(4,4,3,3)$, 
\[
  P(s_{\lambda}\mapsto0)=\begin{cases}
               \frac{q^5+(q-1)q^3}{q^6} = \frac{q^2+q-1}{q^3} &\text{if \ $q \equiv 0 \mod{3}$}\\
               \frac{q^5+(q-1)(q^3-1)}{q^6} &\text{if \ $q \equiv 1 \mod{3}$}\\
               \frac{q^5+(q-1)(q^3+1)}{q^6} &\text{if \ $q \equiv 2 \mod{3}$}
            \end{cases}
\]
and for $\mu=(4,4,3,2)$, 
\[
  P(s_{\lambda_2}\mapsto0)=\begin{cases}
               \frac{q^4+(q-1)(q^2+q-1)}{q^5} &\text{if \ $q \equiv 0 \mod{2}$}\\
               \frac{q^4+(q-1)(q^2+q-2)}{q^5}  &\text{if \ $q \equiv 1 \mod{2}$.}
            \end{cases}
\]
Every partition with fewer than four parts yields a rational function. Of all partitions with exactly four parts, these three examples are the smallest that do not yield rational functions in terms of the size of their first part.
\end{rem}

We call a function $g:\Z \to \Z$ is a  \textit{quasi-polynomial} if there exists an integer $N > 0$ and polynomials $ g_0, g_1, \cdots , g_{N-1}\in \Z[x]$ such that $g(n) = g_i(n)$ if $n \equiv i \mod{N}$.

\begin{conjecture}
For a partition $\lambda$, $P(s_\lambda\mapsto0)$ is always in the form $f(q)/q^k$, where $k$ is a positive integer and $f(q)$ is a quasi-polynomial depending on the residue class of $q$ modulo some integer.
\end{conjecture}

\subsection{Conjecture on the Upper Bound}\

For any given $k$, we next consider the special case where the smallest part of $\lambda$ is at least $k$ and the row sizes are far apart enough that there are no constant or repeated entries in the Jacobi-Trudi matrix. We then remove the condition that the smallest part is at least $k$.

\begin{prop} \label{prop: general shape} Let $\lambda = (\lambda_1,\ldots,\lambda_k)$, where $\lambda_i - \lambda_{i + 1} \geq k - 1$ and $\lambda_k \geq k$. Then
$$P(s_{\lambda} \mapsto 0) = 1- \frac{GL_{k}(\F_q)}{q^{k^2}}= \frac{1}{q^{k^2}}\left(q^{k^2} - \prod_{j=0}^{k-1}(q^k - q^j)\right).$$
All the other conditions being equal, if we have $\lambda_k < k$ instead, then
$$P(s_{\lambda} \mapsto 0) = 1- \frac{GL_{k-1}(\F_q)}{q^{(k-1)^2}}= \frac{1}{q^{(k-1)^2}}\left(q^{(k-1)^2} - \prod_{j=0}^{k-2}(q^{k-1} - q^j)\right)$$
\end{prop}

\begin{proof} We have
$$s_{\lambda} = \begin{vmatrix}
h_{\lambda_1} & h_{\lambda_1 + 1} & \cdots & h_{\lambda_1 + k - 1} \\
h_{\lambda_2 - 1} & h_{\lambda_2} & \cdots & h_{\lambda_2+k-2} \\
\vdots & & \ddots & \vdots \\
h_{\lambda_k - k + 1} & & \cdots & h_{\lambda_k} \\
\end{vmatrix}.$$

The first assertion follows immediately from the fact that there is no constant or repeated variable in the Jacobi-Trudi matrix for $s_{\lambda}$. 

For the second assertion, we will count the number of invertible matrices of the given form. 
Since we have $\lambda_{k-1} - \lambda_{k} \geq k-1$, $$\lambda_{k-1} \geq \lambda_{k} + k - 1 \geq 1+k-1 \geq k.$$ Hence the first $k-1$ rows only have $h_j$'s in them, while the last row has $\lambda_k$ instances of $h_j$'s, one 1, and $k-1-\lambda_k$ zeros. Since row $k$ has a nonzero entry and thus cannot be the zero vector, we may choose the $\lambda_k$ non-constant entries in row $k$ freely. Thus there are $q^{\lambda_k}$ ways to choose row $k$. We can choose the other rows from bottom to top successively, and for row $k-i$ where $1 \leq <k$, there are $q^k - q^i$ ways to choose so that row $k-i$ is not in the span of the rows previously chosen. So we have in total $q^{\lambda_k}\prod_{i=1}^{k-1}(q^k - q^i)$ ways to obtain an invertible matrix. There are $(k-1)k+\lambda_k$ distinct $h_j$'s, so there are in total $q^{(k-1)k+\lambda_k}$ matrices.
Therefore
  \begin{align*}
    P(s_{\lambda} \mapsto 0) &= 1 - \frac{q^{\lambda_k}\prod_{i=1}^{k-1}(q^k - q^i)}{q^{(k-1)k+\lambda_k}}  \\
    &= \frac{1}{q^{k^2-k}}\left( q^{k^2-k} - \prod_{i=1}^{k-1}(q^k - q^i) \right)\\
    &= \frac{1}{q^{k^2-k}}\left( q^{k^2-k} - q^{k-1}\prod_{i=0}^{k-2}(q^{k-1} - q^i) \right)\\
    &= \frac{1}{q^{(k-1)^2}}\left(q^{(k-1)^2} - \prod_{i=0}^{k-2}(q^{k-1} - q^i)\right) \\
    &= 1 - \frac{|GL_{k-1}(\F_q)|}{q^{(k-1)^2}}.
  \end{align*}
\end{proof}

Conceivably, if we have repeated variables or $1$'s in the Jacobi-Trudi matrix, then it would be harder to make some of the rows to be linearly dependent, which in turns decreases probability that the Schur function is mapped to zero. This suggests that for a partition $\lambda$ with $k$ parts, $1- GL_{k}(\F_q)/q^{k^2}$ provides an upper bound for the probability $P(s_{\lambda} \mapsto 0)$. Numerical calculations seem to support this. 
\begin{conjecture}[Upper Bound]
For any partition $\lambda=(\lambda_1,\ldots,\lambda_k)$, we have $$P(s_{\lambda} \mapsto 0) \leq 1- GL_{k}(\F_q)/q^{k^2}.$$
\end{conjecture}

\section{Rectangles and Staircases}\label{sec:rectangles}
In Section~\ref{hooks_section}, we showed that for any hook shape $\lambda$, $P(s_\lambda \mapsto 0)=1/q$. We next show that the same is true for rectangles and staircases.

\subsection{Rectangles}

Let $\lambda= (a^n) $ be a rectangle, where integers $a,n\geq1$. We will deal explicitly with the case where $a\geq n$ as the case where $a<n$ will follow from Corollary~\ref{cor:transpose}. Given such a $\lambda$ with $a\geq n\geq1$, we rename the the variables in the corresponding Jacobi-Trudi matrix and write the matrix as $A=(x_{j-i+n})_{1\leq i,j\leq n}$. 

\begin{lem}\label{lem:rec}
Let $A=(x_{j-i+n})_{1\leq i,j\leq n}$ be an $n\times n$ matrix, where $x_1,\ldots,x_{2n-1}$ are variables. Then for any $0\leq r\leq q$ and any $(a_1,\ldots,a_r)\in\F_q^r$, we have three options for $\varphi(A;x_1=a_1,\ldots,x_r=a_r)$:
\begin{enumerate}
\item $\varphi(A;x_1=a_1,\ldots,x_r=a_r)$ is empty,
\item $\varphi(A;x_1=a_1,\ldots,x_r=a_r)$ is a matrix of size $n'\geq 1$ such that each of its first $n'$ diagonals contains only zeroes, or
\item $\varphi(A;x_1=a_1,\ldots,x_r=a_r)$ is a matrix of size $n'\geq 1$ with
$1\leq k\leq n'$ such that all entries in the $i^{\text{th}}$ diagonal are
  0 for $i<k$, all entries in the $k^{\text{th}}$ diagonal are equal, and all
  entries in the $i^{\text{th}}$ diagonal have label $2n-2n'+i$ for $i\geq k$.
\end{enumerate}
\end{lem}

Notice that if the lower left corner of $\varphi(A;x_1=a_1,\ldots,x_r=a_r)$
has positive label, then the lemma holds trivially with $k=1$. 

Intuitively, this lemma is saying that no matter how we assign variables and perform row and column operations, the matrix behaves nicely as shown in the example below. This nice behavior is what drives the proofs in this rectangle case and will be extensively used for the rest of the paper. 

\begin{ex} 
Starting with matrix $A$, we see that at each step, the matrix satisfies Lemma \ref{lem:rec}.
In other words, these matrices have zeroes in the first few diagonals and entries are equal in the first nonzero diagonal. In particular, the matrices fit into option 3 with $k=1$, $k=1$, $k=1$, $k=2$, and $k=3$, respectively.
\begin{align*}
A=&\begin{bmatrix}
x_4 & x_5 & x_6 & x_7 \\
x_3 & x_4 & x_5 & x_6 \\
x_2 & x_3 & x_4 & x_5 \\
x_1 & x_2 & x_3 & x_4
\end{bmatrix}
\xrightarrow{\varphi(x_1=1)}
\begin{bmatrix}
x_5-x_2x_4 & x_6-x_3x_4 & x_7-x_4^2 \\
x_4-x_2x_3 & x_5-x_3^2 & x_6-x_3x_4 \\
x_3-x_2^2 & x_4-x_2x_3 & x_5-x_2x_4 \\
\end{bmatrix}\\
\xrightarrow{\varphi(x_2=2)}&
\begin{bmatrix}
x_5-2x_4 & x_6-x_3x_4 & x_7-x_4^2 \\
x_4-2x_3 & x_5-x_3^2 & x_6-x_3x_4 \\
x_3-4 & x_4-2x_3 & x_5-2x_4 \\
\end{bmatrix}
\xrightarrow{\varphi(x_3=4)}
\begin{bmatrix}
x_5-2x_4 & x_6-4x_4 & x_7-x_4^2 \\
x_4-8 & x_5-16 & x_6-4x_4 \\
0 & x_4-8 & x_5-2x_4 \\
\end{bmatrix}\\
\xrightarrow{\varphi(x_4=8)}&
\begin{bmatrix}
x_5-16 & x_6-32 & x_7-64 \\
0 & x_5-16 & x_6-32 \\
0 & 0 & x_5-16 \\
\end{bmatrix}
\end{align*}
\end{ex}

Suppose that $A$ is an $n\times n$ matrix and that $S$ and $T$ are subsets of $[n]$ of size $k$. Define $A[S,T]$ to be the $k\times k$ submatrix of $A$ obtained by selecting entries in the $i^{\text{th}}$ row and $j^{\text{th}}$ column of $A$ for all $i\in S$ and $j\in T$. The following lemma is immediate.

\begin{lem}\label{lem:minor}
The determinant of $A[S,T]$ does not change if we perform row and column operations on $A$ that only involve subtracting by rows in $S$ and by columns in $T$.
\end{lem}

Now we are ready to prove Lemma \ref{lem:rec}.

\begin{proof}[Proof of Lemma~\ref{lem:rec}]

We will use proof by induction on $r$. When $r=0$, the lemma holds trivially.

Assume that the lemma holds for $r-1$ and consider the matrix $B=\varphi(A;x_1=a_1,\ldots,x_{r-1}=a_{r-1})$ of size $n'$. If $B$ is empty, or equivalently, $n'=0$, then assigning any $a_r\in\F_q$ to $x_r$ will still result in an empty matrix, so the inductive step holds trivially. Similarly, if all entries in $B$ below or on the main diagonal are zero, the inductive step will also hold. Also, if the smallest positive label of entries in $B$ is strictly greater than $r$, assigning $a_r$ to $x_r$ will not change any labels so the inductive step holds. 

We are now left to consider the case where $B$ satisfies option (3). Before
addressing it in full generality, we show the special case $k = n'$, 
which will be useful later. In this case, the main diagonal entries are the same, say of
label $r$. We may assume they are of the form $x_r - b$ for some $b \in
\F_q$. Then looking at $\varphi(B; x_r = a_r)$, we see that if $a_r \neq b$,
the matrix will become empty after applying $\psi$, satisfying option (1). If
$a_r = b$, then the matrix will have only zeros on and below the main
diagonal, satisfying option (2).

Now we consider the case that $B$ satisfies option (3) in full
generality. Suppose that the first nonzero diagonal of $B$ has entries $x_r-b$
by inductive hypothesis, and suppose that the next diagonal contains
$x_{r+1}-f_1(x_r),\ldots,x_{r+1}-f_{\ell}(x_r)$, where $f_i \in \F_q[x_r]$.

%% \textcolor{red}{If the labels of the main diagonal entries are $r$, since these entries are the same, we can assume that they are $x_r-b$ for some $b\in\F_q$. Then after we apply $\varphi(x_r=a_r)$}, if $a_r\neq b$, the matrix will become empty after $\psi$ and if $a_r=b$, the matrix will have only zeroes on the main diagonal and the diagonals below the main diagonal. 

%% Now we consider the main case. Suppose that the first nonzero diagonal of $B$ has entries $x_r-b$ by inductive hypothesis, and suppose that the next diagonal contains $x_{r+1}-f_1,\ldots,x_{r+1}-f_{\ell}$ where $f_i$'s are polynomials in $x_r$. \Becky{They are in $\F_q[x_r]$?}
$$B=\begin{bmatrix}
\vdots & & & & & & \\
x_{r+1}-f_1 & & & & & & \\
x_r-b & x_{r+1}-f_2 & & & & & \\
0 & x_r-b & \rotatebox[origin=c]{-45}{$\cdots$} & & & & \\
\vdots & 0 & \rotatebox[origin=c]{-45}{$\cdots$} & \rotatebox[origin=c]{-45}{$\cdots$} & & \\
0 & \rotatebox[origin=c]{-45}{$\cdots$} & 0 & x_r-b & \rotatebox[origin=c]{-45}{$\cdots$} & & \\
0 & 0 & \cdots & 0 & x_r-b & x_{r+1}-f_{\ell} & \cdots
\end{bmatrix}$$
If we assign $a_{r}\neq b$ to $x_r$ and apply $\psi$ to this matrix, we will
use the $a_r-b$ to cancel out their corresponding rows and columns. The
remaining matrix $\varphi(B;x_r=a_r)$ contains no constant so it falls into option (3) trivially with $k=1$.
Therefore, the only case we need to consider is when
$a_r=b$. It suffices to show that when $x_r=b$, $x_{r+1}-f_i(x_r)$ is the same for
each $i$. 

Let $m=n-n'$, where $n'$ is the size of $B$. According to the definition of $\varphi$ and induction hypothesis, we can use row and column operations involving only subtracting by the last $m$ rows and the first $m$ columns to get from $A':=A(x_1=a_1,\ldots,x_r=b)$ to $\begin{bmatrix}0&B'\\M&0\end{bmatrix}$, where $B'=B(x_r=b)$ and $M$ can be written as blocks of nonzero multiples of identity matrices from southwest to northeast. Specifically,
\begin{equation} \label{eqn:rec}
M=\begin{bmatrix}&&&c_wI_{s_w}\\&&\rotatebox[origin=c]{45}{$\cdots$}&\\&c_2I_{s_2}&&\\c_1I_{s_1}&&&\end{bmatrix},
\end{equation}
where $c_i\in\F_q^{\times}$, $s_i\in\mathbb{N},$ and $I_{s_i}$ is the $s_i\times s_i$ identity matrix. 
We know that $M$ is $m\times m$, and we let $d=\det M\neq0$.

Notice that the first $k$ diagonals of $B'$ are all zero and the $(k+1)^{\text{th}}$ diagonal of $B'$ has entries of label $r+1$ where $k\leq n-m=n'$. We have already proved the case of $k=n'$ so now assume that $k<n'$. In other words, $B'_{i,j}=0$ for all $i,j\in\Z_{\geq1}$ such that $i-j\geq n-m-k=n'-k$. We need to show that all entries on the $(k+1)^{\text{th}}$ diagonal are the same.

Let $S=\{n-m+1,n-m+2,\ldots,n-1,n\}$ and $T=\{1,2,\ldots,m\}$ and consider $A'[S,T]$. By Lemma \ref{lem:minor}, the determinant of $A'[S,T]$ does not change as we go from $A'$ to  $\begin{bmatrix}0&B'\\M&0\end{bmatrix},$ so $\det A'[S,T]=\det M=d\neq0$.

Since $\det A'[S,T]\neq0$, $\{(a_t,a_{t+1},\ldots,a_{t+m-1})^T:t=1,\ldots,m\}$ forms a linear basis of $\F_q^m$. It follows that there exists a row vector $\V\in\F_q^m$ such that $\V\cdot A'[S,T]=(a_{m+1},a_{m+2},\ldots,a_{2m})$.

We know that for any $j=1,\ldots,k$, by Lemma~\ref{lem:minor}, the determinant of $A'\big[S\cup\{n-m\},T\cup\{m+j\}\big]$ does not change as we go from $A'$ to  $\begin{bmatrix}0&B'\\M&0\end{bmatrix}$. As a minor of the latter matrix, every entry of the first row is zero possibly except the last entry, which is $B'_{n',j}$. Since $n'-j\geq n'-k$, it is also zero. As a result, $\det A'\big[S\cup\{n-m\},T\cup\{m+j\}\big]=0$. 
Fix one such $j\in\{1,\ldots,k\}$, and note that we now have an $(m+1)\times(m+1)$ submatrix of $A'$ with zero determinant, which is $A'\big[S\cup\{n-m\},T\cup\{m+j\}\big]$.

Subtract $\V\cdot A'\big[S,T\cup\{m+1\}\big]$---a linear combination of the last $m$ rows---from the first row of $A'\big[S\cup\{n-m\},T\cup\{m+j\}\big]$. The first $m$ entries of the first row of the resulting matrix are 0 and the last entry in the first row is $a_{2m+j}-\V\cdot(a_{m+j},a_{m+j-1},\cdots,a_{2m+j-1})^T$.

%if we subtract the first row by $\V\cdot A'\big[S,T\cup\{m+1\}\big]$, which is a linear combination of the last $m$ rows, the first row becomes 0 for the first $m$ entries and $a_{2m+j}-\V\cdot(a_{m+j},a_{m+j-1},\cdots,a_{2m+j-1})$ for the last entry. 

Since the resulting $(m+1)\times(m+1)$
submatrix still has determinant zero, we have $$\det A'[S,T]\Big(a_{2m+j}-\V\cdot(a_{m+j},a_{m+j-1},\cdots,a_{2m+j-1})^T\Big)=0.$$
As $\det A'[S,T]=d\neq0$, we know $a_{2m+j}=\V\cdot(a_{m+j},a_{m+j-1},\cdots,a_{2m+j-1})^T.$ With the goal of expressing this observation using matrix multiplication, define $m\times m$ matrix $V$ shown below
$$
\newcommand*{\temp}{\multicolumn{1}{|c}{}}
\newcommand*{\temps}{\multicolumn{1}{c}{\ \ \ \ }}
V:=\left(
\begin{array}{cccc}
 & \temp & & \\
0 & \temp & I_{m-1} & \\ 
\temps & \temp & & \\ \hline
 & \V &  & 
\end{array}\right)
$$
where the ``0" here represents an $(m-1)\times1$ array with all zeroes. We may then write
$$V\cdot(a_{t},a_{t+1},\ldots,a_{t+m-1})^T=(a_{t+1},a_{t+2},\ldots,a_{t+m})^T\text{ for all }1\leq t\leq m+k.$$
A simple inductive argument will give us, for all $t+h\leq m+k+1$,
\begin{equation}\label{eqn:linear1}
V^h\cdot(a_{t},a_{t+1},\ldots,a_{t+m-1})^T=(a_{t+h},a_{t+h+1},\ldots,a_{t+h+m-1})^T.
\end{equation}
At the same time, for all $t+h\leq m+k$, $t\geq1$, $h\geq0$, we have
\begin{equation}\label{eqn:linear2}
\V\cdot V^h\cdot(a_{t},a_{t+1},\ldots,a_{t+m-1})^T=\V\cdot(a_{t+h},a_{t+h+1},\ldots,a_{t+h+m-1})^T=a_{t+h+m}.
\end{equation}

Now we have tools we need to deal with $B'_{i,j}$ for $i-j=n-m-k-1$. Recall that $k<n'$ so $i-j=n'-k-1\geq0$. Fix such $i$ and $j$. For convenience of notation, let $i'=n-m-i+1=n'+1-i\geq1$ so we have $i'+j=k+2$ with $i,j\in\{1,\ldots,k+1\}$. %\Becky{Can we just write $i'=n-m-i$? instead of reusing $i$ to mean something different? this is confusing.}\Yibo{Yeah.} 
and deal with $B'_{n-m+1-i',j}$. Correspondingly, consider $A'\big[S\cup\{n-m+1-i'\},T\cup\{m+j\}\big]$. To calculate the determinant of this $(m+1)\times(m+1)$ submatrix, we subtract the first row by $$(\V V^{i'-1})\cdot A'\big[S,T\cup\{m+j\}\big],$$ which is a linear combination of its last $m$ rows. Then, the first $m$ entries of its first row become 0 by Equation \eqref{eqn:linear2} with $h=i'-1$ and $t=1,\ldots,m$, noticing that indices are in the proper range since $m+i'-1\leq m+k$. The last entry of the first row is
\begin{equation}\label{eqn:linear3}
\begin{split}
&x_{2m+k+1}-v\cdot V^{i'-1}\cdot(a_{m+j},a_{m+j-1},\ldots,a_{2m+j-1})^T\\
=&x_{2m+k+1}-v\cdot V^{i'-1}\cdot V^{m+j-1}\cdot(a_1,a_2,\ldots,a_m)^T\\
=&x_{2m+k+1}-v\cdot V^{m+k}\cdot(a_1,a_2,\ldots,a_m)^T
\end{split}
\end{equation}
where the first equality comes from Equation \eqref{eqn:linear1} with $t=1$ and $h=m+j-1\leq m+k$. 

With this, it is clear that 
\begin{align*}
\det A'\big[S\cup\{n-m+1-i'\},T\cup\{m+j\}\big]=&\det A'[S,T]\cdot\big(x_{2m+k+1}-v\cdot V^{m+k}\cdot(a_1,a_2,\ldots,a_m)^T\big) \\
=&d\big(x_{2m+k+1}-v\cdot V^{m+k}\cdot(a_1,a_2,\ldots,a_m)^T\big).
\end{align*}

Then, according to the form of $M$ and $B'$ and Lemma \ref{lem:minor},
\begin{align*}
B'_{i,j}=&B'_{n-m+1-i',j}=\frac{\det A'\big[S\cup\{n-m+1-i'\},T\cup\{m+j\}\big]}{\det M}\\
=&x_{2m+k+1}-v\cdot V^{m+k}\cdot(a_1,a_2,\ldots,a_m)^T,
\end{align*}
which is an expression that is independent of $i$ and $j$ (only dependent on $i'+j=k+2$).

A simple index counting shows that $r+1=2m+k+1$. And thus, the inductive step is complete as desired.
\end{proof}

\begin{cor} \label{rectangle_probability}
For any rectangular shape $\lambda$, $P(s_\lambda\mapsto0)=1/q$.
\end{cor}
\begin{proof}
We will use an inductive argument. Let $A$ be the Jacobi-Trudi matrix corresponding to shape $\lambda=(b^n)$. By Corollary \ref{cor:transpose}, we can assume $b \geq n$. We claim that if $B=\varphi(A;x_1=a_1,\ldots,x_r=a_r)$ is not empty and that if its entries in and below the main diagonal are not all 0, then $P(\det B\mapsto 0)=1/q$.

We will induct on the number of free variables, i.e. $2n-1-r$. If $B$ has exactly one free variable, then $B$ must be size 1 to satisfy the above hypotheses, so $P(\det B \mapsto 0) = 1/q$ as desired.

Otherwise, suppose that $B$ has size $n'$. According to Lemma \ref{lem:rec}, if $B$ is not empty and if entries on and below the main diagonal are not all zero, then there exists $0 \leq k\leq 2n'$ such that all entries in the $i^{\text{th}}$ diagonal are zero for all $i<k$ and all entries in the $k^{\text{th}}$ diagonal are equal. Clearly we also have $k \leq n'$. 

All entries on the $k^{\text{th}}$ diagonal are equal to $x_t-f_{t-1}$ for some $t>0$ and some $f_{t-1}$ a polynomial of $x_1,\ldots,x_{t-1}$.

If $k=n'$, then $\det B=(x_t-f_{t-1})^{n'}$ and clearly $\det B=0$ iff $x_t=f_{t-1}$, giving a probability of $1/q$ as desired. If $k<n'$, then 
$$P(\det B\mapsto0)=\frac{1}{q}\sum_{c\in \F_q}P\big(\varphi(A;x_1=a_1,\ldots,x_r=a_r,x_{r+1}=c)\mapsto0\big).$$
We can use the induction hypothesis on each term of the sum because 

\begin{enumerate}
\item $B' = \varphi(A;x_1=a_1,\ldots,x_r=a_r,x_{r+1}=c)$ has fewer free variables than $B$; 
\item $B'$ is not empty since at most $k<n'$ rows and columns will be removed due to row and column operations from $B$; 
\item One of the following holds:
\begin{enumerate}
\item $r+1 < t$, in which case $B'$ is the same size as $B$ and the $k^{th}$ diagonal remains the first nonzero diagonal.
\item $r+1 = t$ and $c=f_{t-1}$, in which case $B'$ is the same size as $B$ and the first nonzero diagonal is $k+1 \leq  n'$.
\item $r+1 = t$ and $c\neq f_{t-1}$, in which case $B'$ is smaller than $B$ with no zeroes.
\end{enumerate}
\end{enumerate}

Thus, $P(\det B\mapsto0)=\frac{1}{q}(q\cdot\frac{1}{q})=1/q$.
\end{proof}

\subsection{Staircases}

Now we turn our attention to the staircase shapes, $\lambda=(k,k-1,\ldots,1)$.
\begin{thm} \label{thm: staircase 1/q}
Let $\lambda=(k,k-1,\ldots,1)$ be a staircase, then
$$P(s_\lambda\mapsto0)=1/q.$$
\end{thm}
\begin{proof}
By Jacobi-Trudi, 
\begin{align*}
s_\lambda &= \begin{vmatrix}
h_k    & \cdots & h_{2k-4} &h_{2k-3} & h_{2k-2} & h_{2k-1} \\
\vdots &        &          & \ddots   &          & \vdots \\
0      & \cdots & h_2      & h_3      & h_4      & h_5 \\
0      & \cdots & 1        & h_1      & h_2      & h_3 \\
0      & \cdots & 0        & 0        & 1        & h_1 \\
\end{vmatrix} \\
&= (-1)^l\ \begin{array}{|c c c ;{1.5pt/3pt} c c c|}
\cdots & h_{2k-4} & h_{2k-2} & \cdots  & h_{2k-3} & h_{2k-1} \\
\vdots & \ddots   & \vdots   & \vdots  & \ddots   & \vdots \\
\cdots & h_2      & h_4      & \cdots  & h_3      & h_5 \\
\cdots & 1        & h_2      & \cdots  & h_1      & h_3 \\
\cdots & 0        & 1        & \cdots  & 0        & h_1 \\
\end{array}.
\end{align*}
where the second equality follows from rearranging the columns so that columns containing a constant 1 are to the left of the vertical line.

Notice that in the new matrix, all the variables on the left side are distinct from all the variables on the right side. Renaming the variables on the left side using $x_i$'s and the variables on the right side using $y_i$'s, we find the above is a determinant of some $(m+n)\times(m+n)$ matrix $M$ of the following form:
$$M =\left[\begin{array}{c c c c c ;{1.5pt/3pt} c c c c c}
x_n     & \cdots & x_{m+n-3} & x_{m+n-2} & x_{m+n-1} & y_{m+1}& \cdots & y_{m+n-2} & y_{m+n-1} & y_{m+n} \\
x_{n-1} &        & x_{m+n-4} & x_{m+n-3} & x_{m+n-2} & y_{m}  &        & y_{m+n-3} & y_{m+n-2} & y_{m+n-1} \\
\vdots  &        & \ddots    &           & \vdots    & \vdots &        & \ddots    &           & \vdots   \\
0       &        & a         & x_1       & x_2       & 0      &        & y_1 & y_2 & y_3 \\
0       &        & 0         & a         & x_1       & 0      &        & 0   & y_1 & y_2 \\
0       & \cdots & 0         & 0         & a         & 0      & \cdots & 0   & 0   & y_1 \\
\end{array}\right]$$
where the left side has $m$ columns, the right side has $n$ columns, and $a \neq 0$ is constant. 

Denote the probability that $M$ is singular by $p(m,n)$. We claim that for all $m \geq 0, n >0$, we have $p(m,n) = 1/q$. We prove this claim by induction on $m$.

If $m=0$, then the matrix is an upper triangular matrix with $n$ instances of $y_1$ on its main diagonal. The determinant is $y_1^n$ and thus $p(0,n) =1/q$ for all $n >0$.

Assume for all $0<k<m$ and $n>0$, $p(k,n)=1/q$. Then consider $p(m,n)$. We have two possible cases: $n \leq m$ and $n > m$.

First notice that if $n \leq m$, then for each $1 \leq i \leq n$, we can subtract $y_1/a$ times the $i^\text{th}$ column from the $(i+m)^\text{th}$ column to obtain a determinant of the form

%First notice that if $n \leq m$, then for each $1 \leq i \leq n$, we can %subtract the $i^\text{th}$ column from the right in the right side by %$y_1/a$ times the $i^\text{th}$ column from the right in the left side, %and obtain a determinant in the form
\begin{align*}
 &\begin{array}{|c c c c ;{1.5pt/3pt} c c c|}
x_n     & \cdots & x_{m+n-2} & x_{m+n-1} & \cdots & y_{m+n-1}-\frac{y_1x_{m+n-2}}{a} & y_{m+n}-\frac{y_1x_{m+n-1}}{a} \\
x_{n-1} &        & x_{m+n-3} & x_{m+n-2} &        & y_{m+n-2}-\frac{y_1x_{m+n-3}}{a} & y_{m+n-1}-\frac{y_1x_{m+n-2}}{a} \\
\vdots  & \ddots &           & \vdots    & \ddots &           & \vdots   \\
0       &        & x_1       & x_2       &        & y_2-\frac{y_1x_1}{a} & y_3-\frac{y_1x_2}{a}  \\
0       &        & a         & x_1       &        & 0   & y_2-\frac{y_1x_1}{a} \\
0       & \cdots & 0         & a         & \cdots & 0   & 0 \\
\end{array} \\
&= (-1)^{n} a \cdot \begin{array}{|c c c ;{1.5pt/3pt} c c c|}
x_n     & \cdots & x_{m+n-2}  & \cdots & y_{m+n-2}' & y_{m+n-1}' \\
x_{n-1} &        & x_{m+n-3}  &        & y_{m+n-3}' & y_{m+n-2}'\\
\vdots  & \ddots &           & \ddots &           & \vdots   \\
0       &        & x_1       &        & y_1' & y_2'  \\
0       &        & a         &        & 0   & y_1' \\
\end{array}
\end{align*}
by expanding across the last row and renaming the variables. We are allowed to rename the $y_i$'s as their values are independent of each other. In this way we obtain a smaller matrix of the same kind with $m-1$ columns on the left side and $n$ columns on the right side. So for $n \leq m$, $p(m,n) = p(m-1,n) = 1/q$ by the inductive hypothesis, and we are left with the case when $n > m$. 

If $y_1 = 0$, then the last row of $M$ will be $(0,\ldots,a,0,\ldots,0)$ with $a$ in the $m^\text{th}$ entry and $0$ in the remaining entries. Expand across the last row and we obtain $(-1)^{n} a$ times a smaller determinant of this kind with $m-1$ columns on the left side and $n$ columns on the right side. Since $a$ is nonzero, we have 

\[P(\det{M} =0 \mid y_1=0) =p(m-1,n) = 1/q\]
by the inductive hypothesis.

If $y_1 \neq 0$, then let $b=y_1$. Using a similar argument as in the case $n \leq m$, interchanging the role of $m$ and $n$ and letting $b$ play the role of $a$, we obtain that the conditional probability that $M$ is singular is given by 

\[P(\det{M}=0 \mid y_1\neq 0)=p(m,n-1)=\cdots =p(m,m).\]
Then it is reduced to the case when $n \leq m$ and $P(\det{M}=0 \mid y_1\neq 0)=1/q$.

So in the case when $n > m$, we have 

\begin{align*}
p(m,n)&=P(\det{M}=0 \mid y_1=0)P(y_1=0) + P(\det{M}=0 \mid y_1\neq 0)P(y_1\neq 0) \\
&= 1/q\cdot 1/q + (q-1)/q\cdot 1/q\\ 
&= 1/q.
\end{align*}

We conclude that $p(m,n)=1/q$ for all $m,n$. In particular, we have \[P(s_\lambda \mapsto 0) = p(\lfloor\frac{k}{2}\rfloor, \lceil\frac{k}{2}\rceil)= 1/q.\]
\end{proof}

\section{Classification of $1/q$}\label{sec:classification}
In this section we prove that the partition shapes that achieve the lower bound of $1/q$ given in Corollary \ref{cor: geq 1/q} are exactly hooks, rectangles, and staircases.

With that goal, we first state and prove a series of lemmas concerning when a reduced general Schur matrix has probability $P(\det{M} \mapsto 0) > 1/q$ as we assign the variables to numbers in $\F_q$ uniformly at random. The key idea is to view the probability $P(\det{M} \mapsto 0)$ as an average of different conditional probabilities coming from partial assignments of variables. By Theorem \ref{thm: general matrix, geq 1/q}, each of these conditional probabilities is at least $1/q$, so our task is reduced to finding a particular conditional probability that is strictly larger than $1/q$.

\begin{lem}\label{lem: general matrix, 1}
Let the matrix $M$ be a reduced general Schur matrix of size $n$ with $m$ variables $x_1, \cdots, x_m$. If either the upper left or lower right entry is zero or the upper left and lower right entries have different labels, then $P(\det{M} \mapsto 0) > 1/q$.
\end{lem}
\begin{proof}
First notice by the definition of general Schur matrix, if an entry is $0$, then all the entries below it or to the left of it must be $0$ as well.
In particular, if either of the upper left or lower right entries is zero, then either the first column or the last row is the zero vector and thus $P(\det{M} = 0) = 1 > 1/q$. So we can assume that neither of these two entries is zero. 

Without loss of generality let the label of the upper left entry be the smaller of the two, and denote it by $k$. We have 
\[P(\det{M} \mapsto 0) = \frac{1}{q^k}\sum_{(a_1,\cdots, a_k) \in \F_q^k} P(\det{M} \mapsto 0 \mid x_1=a_1,\cdots,x_k=a_k).\] 
Notice for any $a_1,\cdots,a_k$, the matrix $M'=\varphi(M; x_1=a_1,\cdots,x_k=a_k)$ is nonempty and thus is a reduced general Schur matrix, since the labels of the entries in the last column of $M$ are all strictly greater than $k$. Hence 
\[P(\det{M} \mapsto 0 \mid x_1=a_1,\cdots,x_k=a_k) = P(\det{M'} \mapsto 0) \geq 1/q\] by Theorem \ref{thm: general matrix, geq 1/q}. Therefore, to show $P(\det{M} \mapsto 0) > 1/q$, it suffices to find one conditional probability that is strictly larger than $1/q$. 

Now we can choose values $b_1,\cdots,b_k$ for $x_1,\cdots,x_k$ such that under this assignment all the entries in the first column become $0$. In this case, the conditional probability $P(\det{M} \mapsto 0 \mid x_1=b_1,\cdots,x_k=b_k) = 1$, as desired.
\end{proof}

\begin{lem}\label{lem: general matrix, 2}
We next strengthen the previous lemma: if the labels of the entries in the first column from bottom to top are not the same as the labels of the entries in the last row from left to right, then $P(\det{M} \mapsto 0) > 1/q$.
\end{lem}
\begin{proof}
Using Lemma \ref{lem: general matrix, 1}, we may assume the upper left entry and the lower right entry are both nonzero and have the same label $l$.  Let $k$ be the smallest label that appears in the first column or last row but not both. Without loss of generality, assume $k$ appears in the first column. Note that the labels of the entries in the first column below $k$ must match the labels of the initial entries of the last row. We have 
\[P(\det{M} \mapsto 0) = \frac{1}{q^{k}}\sum_{(a_1,\cdots, a_{k}) \in \F_q^k} P(\det{M} \mapsto 0 \mid x_1=a_1,\cdots,x_{k}=a_{k}).\]
For any $a_1,\cdots,a_{k}\in \F_q$, the matrix $M'=\varphi(M; x_1=a_1,\cdots,x_{k}=a_{k})$ is nonempty and thus a reduced general Schur matrix, since the labels of the entries in the last column of $M$ are all strictly greater than $k$. Again, we show that $P(\det\varphi(M; x_1=a_1,\cdots,x_{k}=a_{k})\mapsto 0)>1/q$ for some choice of $a_1,\ldots,a_k$.  

Now assign values $b_1,\cdots,b_{k-1}\in \F_q$ to $x_1, \cdots x_{k-1}$  such that all the entries in the first column below the one with label $k$ are zero. Let $M'_i =\varphi(M; x_1=b_1,\cdots,x_{i}=b_{i})$ for $1 \leq i \leq k-1$. If in this process any entry in the last row becomes a nonzero constant, the last row of $M$ is deleted in the process of calculating $M'_{k-1} =\varphi(M; x_1=b_1,\cdots,x_{k-1}=b_{k-1})$. Suppose we delete the last row when assigning $x_t=b_t$ $(t<k-1)$. Denote the label of the entry that lies strictly above the lower right entry in $M'_{t-1}$ as $l'$. We immediately have $l'>l$. After performing $\varphi(M'_{t-1}; x_t=b_t)$, the entry with label $l'$ will become  the lower right entry in $M'_t$, while the upper left entry of $M'_t$ still has label $l$ because we never delete the first column. Since we can only increase the label of the lower right entry in $M'_i$ for $t \leq i \leq k-1$, we know the labels of the upper left and lower right entries in $M'_{k-1}$ do not match. Hence we have $P(\det{M'_{k-1}} \mapsto 0) > 1/q$ from Lemma~\ref{lem: general matrix, 1}. It follows that 

\begin{align*}
\sum_{b_k\in \F_q}P(\det{M} \mapsto 0 \mid x_1=b_1,\cdots,x_{k}=b_{k}) &= P(\det{M} \mapsto 0 \mid x_1=b_1,\cdots,x_{k-1}=b_{k-1}) \\
&= P(\det{M'_{k-1}} \mapsto 0) \\
&> 1/q.
\end{align*}\

If instead there is no nonzero constant in the first row, then we assign some value $b_k$ to $x_k$ to make the entry in the first column with label $k$ nonzero. Then look at $M'_a = \varphi(M; x_1=a_1,\cdots,x_{k-1}=a_{k-1})$. Now the upper left entry in $M'_a$ has some label strictly larger than $l$ while the lower right entry of $M'_a$ still has label $l$. Applying Lemma \ref{lem: general matrix, 1},  \[P(\det{M} \mapsto 0 | x_1=b_1,\cdots,x_{k}=b_{k}) = P(\det{M'_a} \mapsto 0) > 1/q.\]
\end{proof}
\begin{lem}\label{lem: general matrix, 3}
Let $M$ be an $n\times n$ reduced general Schur matrix with $m$ variables $x_1, \cdots, x_m$ with the following properties:

(i) All the entries on the main diagonal are nonzero.

(ii) The label of the upper left and the lower right entries are equal.

(iii) The label of any entry on the main diagonal is at least the label of the upper left entry.
If the labels of entries on the main diagonal of M are not all equal, $P(\det{M} \mapsto 0) > 1/q$.
\end{lem}
\begin{proof}
Denote the label of the the upper left entry by $k$. For any assignment of $x_1=a_1,\cdots,x_{k-1}=a_{k-1}$, the matrix $M'=\varphi(M; x_1=a_1,\cdots,x_{k-1}=a_{k-1})$ is nonempty and thus a reduced general Schur matrix since the labels of the entries in the last column of $M$ are all strictly greater than $k-1$. Hence 
\[P(\det{M} \mapsto 0 \mid x_1=a_1,\cdots,x_{k-1}=a_{k-1}) = P(\det{M'} \mapsto 0) \geq 1/q\]
by Theorem \ref{thm: general matrix, geq 1/q}. Again, it suffices to find one conditional probability larger than $1/q$. 

Assign values $b_1,\cdots,b_{k-1}\in \F_q$ to $x_1,\cdots, x_{k-1}$ so that all entries in the first column except the topmost entry become zero. If in this process we have some nonzero constant in the last row, then by Lemma \ref{lem: general matrix, 2} we already have $P(\det{M} \mapsto 0) > 1/q$. Hence we can assume all except the rightmost entry in the last row also become zero. Let $M' = \varphi(M; x_1=b_1,\cdots,x_{k-1}=b_{k-1})$. Notice in $M'$ the labels of the upper left and lower right entries are still $k$ while labels of the other diagonal entries get larger or stay the same compared to $M$. So the new matrix $M'$ still satisfies the assumption of the lemma. 

Consider the upper left entry of $M'$ and assign $x_k$. There is a $1/q$ chance this assignment makes the upper left entry 0, in which case we have a zero column and the determinant is zero. Otherwise, we apply $\psi$ to obtain a reduced general Schur matrix $M''$ of smaller size. Notice that since there are some entries on the diagonal with label strictly larger than $k$, the number of nonzero constants in $M'$ must be strictly smaller than the its size, so $M''$ is nonempty. By Theorem \ref{thm: general matrix, geq 1/q}, $P(\det{M''}\mapsto0)\geq 1/q$. Combining, we get the conditional probability is 
\begin{align*}
P(\det{M} \mapsto 0 | x_1=b_1,\cdots,x_k=b_k) &= P(\det{M'} \mapsto 0) \\
&= 1/q \cdot 1 + (q-1)/q \cdot P(\det{M''}\mapsto0) \\
&\geq 1/q + (q-1)/q \cdot 1/q\\ 
&> 1/q.
\end{align*}
\end{proof}

Next we turn our attention to special Schur matrices and state and prove some necessary conditions for a special Schur matrix $M$ to have $P(\det{M} \mapsto 0) = 1/q$. The idea is to use proof by contradiction and find one conditional probability strictly larger than $1/q$.

\begin{lem}\label{lem: special matrix, 1}
Let $M$ be a special Schur matrix of size n. If $P(\det{M} \mapsto 0) = 1/q$, then the entries on the main diagonal of $M$ all have the same label.
\end{lem}
\begin{proof}
Assume to the contrary that the entries on the main diagonal of $M$ do not all have the same label, and let $k$ be the largest such label. Assign $0$ to $x_1,\cdots, x_{k-1}$, causing every entry with label smaller than $k$ to become 0.  Then every entry below the diagonal and at least one entry on the diagonal is $0$. Thus
\[P(\det{M} \mapsto 0 | x_1=0,\cdots,x_{k-1}=0) = 1 > 1/q.
\]
Hence $P(\det{M} \mapsto 0) > 1/q$, a contradiction. 
\end{proof}

\begin{lem} \label{lem: special matrix, 2}
Let $M$ be special Schur matrix $M$ of size $n$ with $P(\det{M} \mapsto 0) = 1/q$. Then for each diagonal below the main diagonal, the labels of the northwesternmost entry and southeasternmost entry are equal and are weakly smaller than the label of any other entry on that diagonal. %the smallest in that diagonal.
\end{lem}

\begin{proof}
Assume not, and let $i<n$ be the smallest integer such that the label of the two entries $M_{n-i+1,1}$ and $M_{n,i}$ is not the smallest on the $i^\text{th}$ diagonal. (These two entries have the same label by Lemma \ref{lem: general matrix, 2}.) Let $k$ be the smallest of all the labels on this diagonal. Then by definition of general Schur matrix, entries above the $i^\text{th}$ diagonal all have labels strictly larger than $k$, and by our assumption $M_{n-i+1,1}$ and $M_{n,i}$ also have labels strictly larger than $k$. This shows that the entries with label $k$ must be in the first $i$ diagonals, and on the $i^{\text{th}}$ diagonal there are at most $i-2$ entries with label $k$. If there is some entry $M_{st}$ with label $k$ on the $l$th diagonal where $l<i$, by definition of general Schur matrix we can show that the entries between $M_{n-i+t,t}$ and $M_{s,s+i-n}$ on the $i$th diagonal will all have label greater than $k$. Therefore, the total number of entries in $M$ with label $k$ is also at most $i-2 \leq n-3$. 

Let $M'=\varphi(M;x_1=\cdots=x_{k-1}=0,x_k=1)$. Recall in the application of $\varphi$ rows and columns are only deleted after $x_k$ is assigned. Since in $M$ the number of entries with label $k$ is at most $n-3$, the size of $M'$ is at least $3$. Lemma \ref{lem: special matrix, 1} gives that all the entries on the main diagonal of $M$ have the same label. Now the upper left entry and the lower right entry of $M'$ still have the same label as they come from the corresponding entries of $M$, but some other entries on the main diagonal have strictly larger labels as some columns are shifted to the left. By Lemma \ref{lem: general matrix, 3}, we have $P(\det{M} \mapsto 0) > 1/q$, a contradiction.
\end{proof}

\begin{lem}\label{lem: special matrix, 3}
Let M be a special Schur matrix $M$ of size $n$ with $P(\det{M} \mapsto 0) = 1/q$. Then all the entries on the $(n-1)^\text{th}$ diagonal have the same label.
\end{lem}
\begin{proof}
 Assume not. Denote the largest label among the entries on the $(n-1)^\text{th}$ diagonal by $k$. Since $M_{2,1}$ and $M_{n,n-1}$ have the smallest label on the $(n-1)^\text{th}$ diagonal and the labels of the main diagonal entries are all equal and strictly larger than $k$, there are at most $n-3$ entries in $M$ with label $k$. Let $M'=\varphi(M;x_1=\cdots=x_{k-1}=0,x_k=1)$. Notice the size $M'$ is at least $3$ and in $M'$ the labels of upper left and lower right entries are equal as in $M$, while the labels of some other main diagonal entries are increased. So on the main diagonal, the labels of the upper left and the lower right entries are the smallest and some other entries have strictly larger labels. By Lemma \ref{lem: general matrix, 3}, we have $P(\det{M} \mapsto 0) > 1/q$, a contraction.
\end{proof}
\begin{cor}\label{cor: special matrix, 4}
Let $M$ be a special Schur matrix of size $n$. If we have $P(\det{M} \mapsto 0) = 1/q$, then on each diagonal, the labels of the entries are equal. Further, the labels of each diagonal form an arithmetic progression.
\end{cor}
\begin{proof}
Use Lemma \ref{lem: special matrix, 1} and \ref{lem: special matrix, 3} and induct from the main diagonal down to the first diagonal and up to the $(2n-1)^\text{th}$ diagonal. Use the fact that by definition, for any $2\times2$ submatrix of a special Schur matrix, the sum of labels of the two entries on the diagonal equals the sum of labels of the two entries on the antidiagonal.
\end{proof}

With these lemmas in hand, we can narrow our attention to Jacobi-Trudi matrices. For the rest of this section, let $M$ be the Jacobi-Trudi matrix of a partition shape $\lambda = (\lambda_1, \lambda_2, \ldots, \lambda_n)$. We introduce the following main theorem which characterizes all the possible shapes $\lambda$ that have probability $1/q$. 

\begin{thm}\label{thm:1/q}
Suppose $P(s_\lambda\mapsto 0)=1/q$. Then $\lambda$ is either a hook, a rectangle, or a staircase.
\end{thm}

Denote the Jacobi-Trudi matrix for $s_\lambda$ by $M$. We divide the proof of this theorem into cases.

\textbf{Case 1:} No entry of $M$ is constant. This implies that every entry in $M$ is a variable. Therefore $M$ is a special Schur matrix. Then by Corollary~\ref{cor: special matrix, 4}, the labels of the entries on the main diagonal are equal, which implies that $\lambda_1=\lambda_2=\cdots=\lambda_n$. Thus $\lambda$ is a rectangle. 

\textbf{Case 2:} At least one entry of $M$ is constant. In this case we first state some special cases of $M$, and then generalize all the possible shapes that have probability $P(\det{M} \mapsto 0) = 1/q$.

\begin{lem}\label{lem: psi_M_size}
If $M$ contains at least one constant, then $\psi(M)$ is at least $\lambda_n\times\lambda_n$. 
\end{lem}

\begin{proof}
Since $M$ contains constants, the last row must have at least one constant. Hence $M$ must have at least $\lambda_n$ columns. Since the last $\lambda_n$ columns do not contain any constants, they will not be deleted by the operation $\psi$. Hence $\psi(M)$ is at least $\lambda_n\times\lambda_n$.
\end{proof}

\begin{rem}\label{rem: psi_M_size}
We see that $\psi(M)$ is exactly $\lambda_n\times\lambda_n$ if and only if each of the first $n - \lambda_n$ columns of $M$ contains a nonzero constant, i.e., each of them will be deleted by $\psi$. 
\end{rem}

\begin{lem} \label{lem: 6.8}
If $\lambda_n\geq 2$ and $P(\det{M} \mapsto 0) = 1/q$, then $\lambda = (\lambda_1^{m}, \lambda_n^{n-m})$ where $1\leq m \leq n$. 
\end{lem}

\begin{proof} 
Let $M'=\psi(M)$, and denote the size of $M'$ by $m$. By Corollary $\ref{cor: special matrix, 4}$, the labels of each diagonal of $M'$ form an arithmetic progression. We denote the common difference by $k$. Since $\lambda_n\geq 2$, we have $M_{nn} = h_{\lambda_n}$, $M_{(n-1)n} = h_{\lambda_{n}-1}$, and they are both non-constant. Therefore, the entries $M'_{(m-1)m}$ and $M'_{mm}$ come from the last two columns of $M$, and the difference between their labels is $1$. We hence have $k=1$, which means the labels in each row of $M'$ form a consecutive sequence. It follows that the $1$'s in the original matrix $M$ are in the leftmost $n-m$ rows, otherwise the difference between some adjacent entries in the same row of $M'$ would be at least two. We already know the $1$'s are in the bottom $n-m$ rows of $M$, so we can divide $M$ into four blocks as shown below.

$$\left[
\begin{array}{cccc|cccc}
  h_{\lambda_1} & h_{\lambda_1+1} & \cdots & & h_{a} & h_{a+1} & \cdots \\
  \vdots & h_{\lambda_2} & \cdots & & h_{a-1} & h_{a} & \cdots \\
  & & \ddots & & \vdots & & \ddots \\
  \hline
  1 & h_1 & h_2 & \cdots &  \ddots \\
  0 & 1 & h_1 & \cdots & \\
  \vdots & & \ddots & & & \vdots & \vdots \\
  0 & \cdots & 0 & 1 & \cdots & h_{\lambda_n-1} & h_{\lambda_n}\\
\end{array}
\right]$$

Notice that all the 1's in $M$ must appear consecutively along the diagonal in the lower left block of $M$, and all the entries in $M'$ come from the block $M[[m],\{n-m+1,\cdots,n\}]$ in the upper right corner of $M$. Since $n-m$ is the number of 1's, we have $\lambda_1 = \lambda_2 = \cdots = \lambda_m$ and $\lambda_{m+1}=\lambda_{m+2}=\ldots=\lambda_n$, which gives the desired result. 
\end{proof}

The previous lemma says that if $\lambda =(\lambda_1,\cdots,\lambda_n)$ with $\lambda_n \geq 2$ and $P(s_\lambda \mapsto 0)=1/q$, then $\lambda$ must be a rectangle or a fattened hook. We next show it cannot be the latter. 

% Consider rewriting part of this.
\begin{lem} \label{lem: 6.9}
Let $M$ be the Jacobi-Trudi matrix corresponding to a partition shape $\lambda = (a^p, b^m)$ where $p, m\in \Z^+$  and $a > b \geq 2$. Then $P(\det{M} \mapsto 0) > 1/q$. 
\end{lem}

\begin{proof}
By Corollary $\ref{cor: geq 1/q}$, we only need to find 
some $c_1,\cdots,c_j \in \F_q$ such that $P(M\mapsto 0 \mid h_1=c_1,\cdots,h_j=c_j)>1/q$. 

Let $k=a-b \geq 1$. We can draw the partition shape $\lambda$ as shown in Figure~\ref{fig:lambda1}. It suffices to prove the result for $k \geq m$ by Corollary \ref{cor:transpose}.

\begin{figure}[h!]
\includegraphics[scale=0.25]{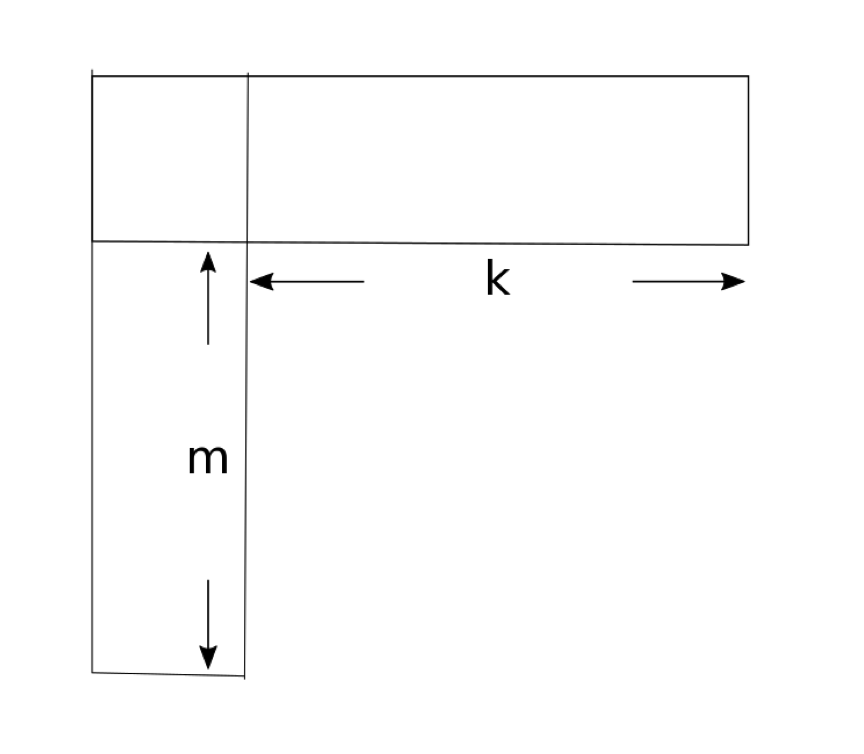}
\caption{Partition shape of $\lambda$}\label{fig:lambda1}
\end{figure} 

The Jacobi-Trudi matrix $M$ is of the following form: 
$$M=\left[
\begin{array}{cccccccc}
  h_{a} & &\cdots &  \\
   & \ddots & \\
\vdots & & h_a & \cdots & \\
 & & \cdots & h_b & \cdots & \\
   1 & \cdots & &&\ddots & & \\
   & \ddots &&&& \ddots \\
   &  & 1 & \cdots &&& h_b
\end{array}
\right].$$

In particular, the main diagonal contains only $h_a$'s and $h_b$'s, and there is an upper triangular matrix with constant $1$ on the main diagonal at the lower left corner. Zooming in, we have the matrix below.

\begin{equation}\label{eqn:M}  
M = \begin{tikzpicture}[baseline={([yshift=0pt]current bounding box.center)},decoration=brace,align=center]
 \matrix(m) [matrix of math nodes,left delimiter=[,right delimiter={]},text width=2.75em] {
{h_a}    & \cdots  & \phantom{|} & \phantom{|} & h_{a+m-1} & h_{a+m} & \phantom{|} & \phantom{|}& \phantom{|} & {\phantom{|}}\\
\vdots &    &    & \ddots         & \vdots   & \vdots      &  & \ddots & & \vdots \\
h_{k+2}     & \cdots & & h_{m+k}      & h_{m+k+1}      & h_{m+k+2}      & \cdots & \\
h_{k+1} & h_{k+2}  &   & \cdots & h_{m+k}  & h_{m+k+1}      & h_{m+k+2}      & \cdots & &  \phantom{|}\\ %\hline
1      & h_1        & h_2  & \cdots  & h_{m-1}  & h_m & h_{m+1} & h_{m+2} & \cdots & h_{m+b-1}\\
0 & 1 & h_1 & \cdots & h_{m-2} & h_{m-1} & h_m & h_{m+1} & \cdots & h_{m+b-2}\\
\vdots &              & \ddots   &      &\vdots    & \vdots & & \vdots \\
0   &   & \cdots & 0        & 1        & h_1 & h_2 &  & \cdots     & {h_b} \\
  };
    \draw[decorate,transform canvas={yshift=0.5em}, thick] (m-1-1.north west) -- node[above=2pt] {$m$ columns} (m-1-5.north east);
    \draw[decorate,transform canvas={yshift=0.5em}, thick] (m-1-6.north west) -- node[above=2pt] {$b$ columns} (m-1-10.north east);
    \draw[decorate,transform canvas={xshift=1.2em}, thick] (m-1-10.north east) -- node[right=2pt] {$b$ rows} (m-4-10.south east);
    \draw[decorate,transform canvas={xshift=1.2em}, thick] (m-5-10.north east) -- node[right=2pt] {$m$ rows} (m-8-10.south east);
    \draw[thick] (m-1-5.north east) -- (m-8-5.south east);
    \draw[thick] (m-5-1.north west) -- (m-5-10.north east);
\end{tikzpicture} 
\end{equation}

Notice neither the block matrix at the top left corner nor the one at the lower right corner is a square matrix. The form of $M$ gives
$$\psi(M) = \begin{bmatrix}
\vdots \\
h_{m+k+2} - G_{m+k+1} & \iddots \\
h_{m+k+1} - F_{m+k} & h_{m+k+2} - H_{m+k+1} & \cdots 
\end{bmatrix}$$
where $\psi(M)$ is $b\times b$, $F_{m+k}$ is a polynomial of $h_1$ through $h_{m+k}$ with zero constant term and $G_{m+k+1},H_{m+k+1}$ are polynomials of $h_1$ through $h_{m+k+1}$ with zero constant terms. 

We assign $h_1=h_2=\cdots=h_m = h_{m+k+1} = 0$ and  $h_{m+1} = h_{k+1} = 1$. Since $k+1>m$, this assignment is valid in the sense that no single variable is assigned two different values. Under this assignment, we have 
$$M' = \left[
\begin{array}{ccccc|cccc}
h_a    & \cdots  & & & &\\
\vdots &    &    &          & \ddots   &          & \vdots \\
h_{k+2}     & \cdots & & h_{m+k}      & 0      & h_{m+k+2}      & \cdots \\
1 & h_{k+2}  &   & \cdots & h_{m+k}      & 0      & h_{m+k+2}      & \cdots \\ \hline
1       & 0        & 0  & \cdots  & 0  & 0 & 1 & h_{m+2} & \cdots \\
0 & 1 & 0 & \cdots & 0 & 0 & 0 & 1 & \cdots \\
\vdots &              & \ddots   &      &\vdots    & \vdots & \vdots \\
0   &   & \cdots & 0        & 1        & 0 & 0 & \cdots     & h_b \\
\end{array}
\right]
$$
%\Becky{The above matrix is actually some $M'$ and below we look at $\psi(M')$? I think it's misleading to call it $M$.}
We can now see that in $\psi(M')$, we have $h_{m+k+1} - F_{m+k} = 0$, $G_{m+k+1}=0$ while $H_{m+k+1} = h_{m+1}h_{k+1} = 1\neq 0$. Therefore by Lemma $\ref{lem: general matrix, 2}$ we have $P(\det{M} \mapsto 0) > 1/q$. 
\end{proof}

\begin{proof}[Proof of Theorem~\ref{thm:1/q}]
 
By Corollary \ref{cor: geq 1/q}, for any shape $\lambda$ that is not a hook, a rectangle or a staircase, we only need to show that there exist some $c_1,\cdots,c_j \in \F_q$ such that $P(M\mapsto 0 \mid h_1=c_1,\cdots,h_j=c_j)>1/q$, where $M$ is its Jacobi-Trudi matrix.

Given a partition shape $\lambda = (\lambda_1, \lambda_2, \ldots, \lambda_n)$ that is not a hook, a rectangle or a staircase, we perform $\psi$ on its Jacobi-Trudi matrix $M$ and denote $M'=\psi(M)$. By Proposition~\ref{prop: psi_M}, $M'$ is a special Schur matrix. By Corollary~\ref{cor: special matrix, 4}, we can assume the labels of each diagonal of $M'$ form an arithmetic progression. We denote the common difference of the labels of two consecutive entries by $k$. 

\textbf{Subcase 1:} $k=0$. This means that $\psi(M)$ is a $1\times1$ matrix. By Lemma~\ref{lem: psi_M_size}, $\lambda_n=1$. By Remark~\ref{rem: psi_M_size}, every column except the last column of $M$ contains 1. This implies that $\lambda_2=\lambda_3=\cdots=\lambda_n=1$. Hence $\lambda$ is a hook. Since we assume that $\lambda$ is not a hook, this case will not happen.

\textbf{Subcase 2: }$k=1$. In order to maintain the probability $1/q$, we claim that $\psi(M)$ must be $\lambda_n\times\lambda_n$. If $\psi(M)$ contains at least $\lambda_n+1$ columns, then the $(\lambda_n+1)^\text{th}$ column of $M'$ counted from the right cannot come from the $(\lambda_n+1)^\text{th}$ column of $M$ from the right since this column contains $1$, so it is at least one column away from the last $\lambda_n$ columns in $M$. In this case, the labels in $\psi(M)$ cannot form an arithmetic progression, so $P(\det{M} \mapsto 0) > 1/q$ by Corollary~\ref{cor: special matrix, 4}. Since $\psi(M)$ is at least $2\times2$, $\lambda_n\geq 2$, and $\lambda$ is not a rectangle, it follows from Lemmas $\ref{lem: 6.8}$ and \ref{lem: 6.9} that $P(\det M \mapsto 0)>1/q$. 

\textbf{Subcase 3: } $k\geq 2$. In Subcase 2 and the proof of Lemma 5.11 we have shown that $k=1$ if and only if $\lambda_n \geq 2$. Hence in this subcase, $\lambda_n=1$, which implies $M_{n(n-1)}=1$. By the fact that  $k$ is the common difference of the labels of two consecutive entries in $M$, there must be $k-1$ columns between any two columns in $M$ that do not contain $1$. This means that between any two columns with no nonzero constants, there are $k-1$ columns that have nonzero constants. The constant 1's must appear consecutively starting from the bottom to the top along some diagonal. Therefore, we can write the lower right corner of $M$ in the following form: 

\begin{align*}
M_{LR} &= \begin{vmatrix}
0 & 1 & h_1 & \cdots & h_{k-2} & h_{k-1} & \cdots & & \\
0 & 0 & 1 && \cdots & h_{k-2} & \cdots \\
&&&\ddots \\
&&&& 1 & h_1 & \cdots\\
&&&&& 0 & 1 & \cdots \\
&&&&&&& \ddots \\
&&&&&&&&1 & h_1
\end{vmatrix}
\end{align*}
where $\psi(M)$ has size $b$ and $M_{LR}$ is a $(kb-k-b+1)\times (kb-b+1)$ matrix. Since $M_{LR}$ includes all the columns not deleted by the operation $\psi$, the rest of the columns must contain constant 1, and they must appear on the left side of $M$. Assume there are $p$ such columns, then similar to Equation~\ref{eqn:M} in Lemma~\ref{lem: 6.9} we can write $M$ in the following form: 

$$
\newcommand*{\temp}{\multicolumn{1}{c|}{}}
M = \left[\begin{array}{cccc|cccccc}
\vdots & & \ddots & & & & & & & \\
h_{r+k+1}      & & \cdots  &  h_{r+p+k}   & h_{r+p+k+1}   &\cdots    \\
h_{r+1}   & \cdots & & h_{r+p}      & h_{r+p+1}   & & \cdots  & & h_{r+p+k+1}   & \cdots \\ \hline
1       & h_1           & \cdots & h_{p-1}   & h_p & h_{p+1} & \cdots & & h_{p+k} & \cdots \\
0 & 1 & h_1 & \vdots & h_{p-1} & h_p  & \cdots & \vdots \\
\vdots &              & \ddots   &          & \vdots & \vdots \\
0      & \cdots & 0   & 1        & h_1 &  h_2 & \cdots     & h_k & h_{k+1} & \cdots \\ \hline
&&&\temp\\
&0&&\temp&&& M_{lr}\\
&&&\temp
\end{array}\right]
$$
%Should the x_k,x_k+1, x_k-1 be 1 row lower in this matrix?
%Also, what about the case when there is a 0 immediately below x_r? I think you cannot assume p>0.
where $r$ is the difference between two parts of the partition.

As shown above, there are $(k-1)(b-1)$ $1$'s in $M_{LR}$. Also notice that starting from the row with $h_{r+1}$ and going to the top, the difference in labels between two consecutive rows is exactly $k$, showing the difference between the corresponding $\lambda_i$ and $\lambda_{i+1}$ is $k-1$. Then we can draw the partition shape $\lambda$ as shown in Figure~\ref{fig:lambda}.

\begin{figure}[h!]
\includegraphics[scale=0.2]{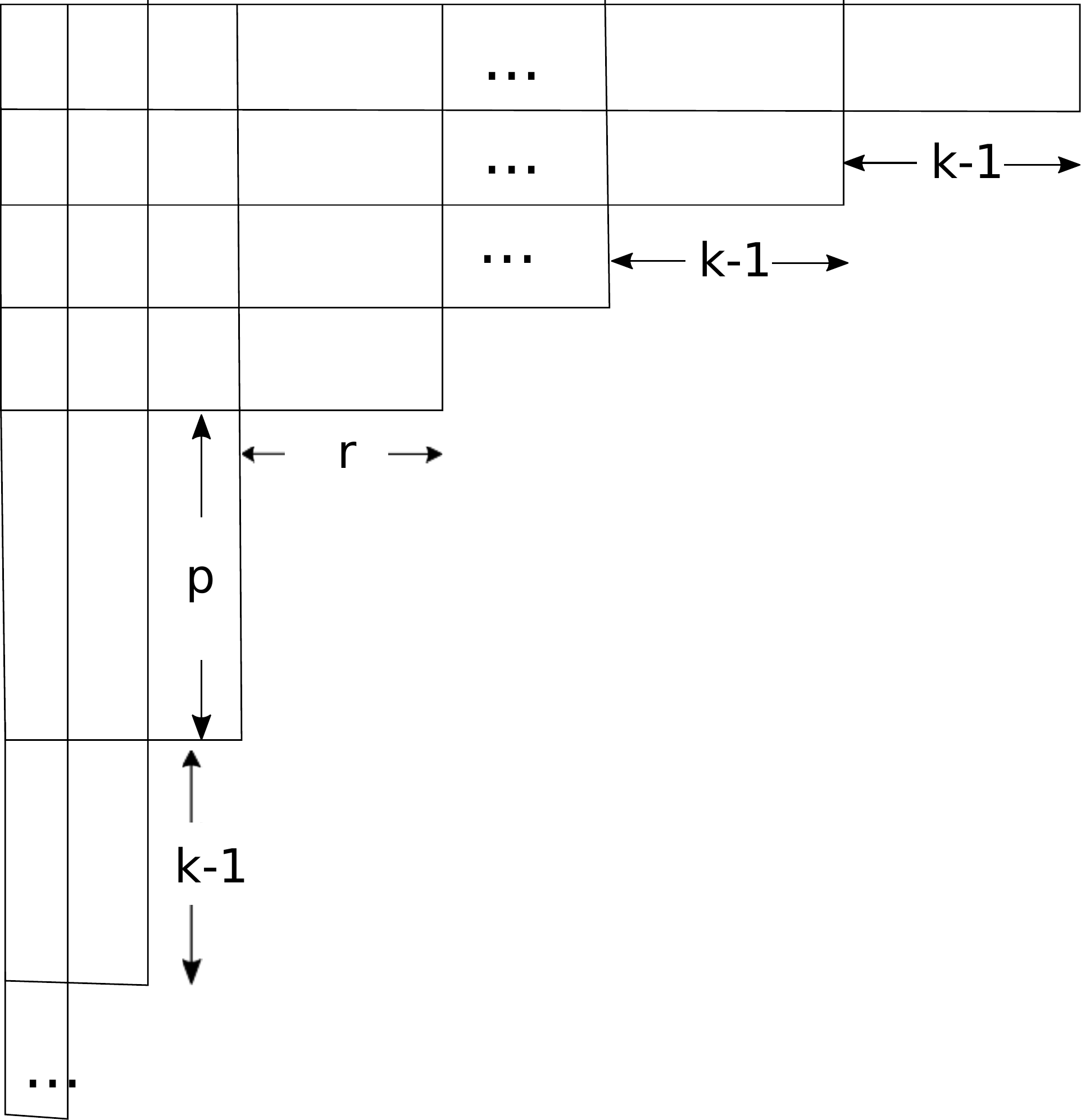}
\caption{Partition shape of $\lambda$}
\label{fig:lambda}
\end{figure}

By Corollary~\ref{cor:transpose} we may assume without loss of generality that $r\geq p$. 
Then 
$$\psi(M)=\begin{vmatrix}
\vdots \\
h_{r+p+k+1} - F_{r+p+k} & \iddots \\
h_{r+p+1} - G_{r+p} & h_{r+p+k+1} - H_{r+p+k} & \cdots 
\end{vmatrix}$$
%h is not a good name for this polynomial, since it is already used for complete homogeneous symmetric polynomials. Maybe use f with superscripts?
where $F_i, G_i, H_i$ are polynomials of $h_1$ through $h_{i}$.

If $p=0$, we assign $h_1=h_2=\cdots=h_{k-2}=0$, $h_{r+1}=0$ and $h_{k-1}=h_{r+2}=1$. Under this assignment, in $\psi(M)$, we have $h_{r+p+1}-G_{r+p}=0$, $F_{r+p+k}=0$, while $H_{r+p+k}=h_{k-1}h_{r+2}=1 \neq 0$. Therefore by Lemma \ref{lem: special matrix, 3} we have $P(\det(M) \mapsto 0) > 1/q$. 

If $p\neq 0, p\neq k-1$, then $p+1 \neq k$, $r+k\neq r+p+1$ and $r+k>r+1\geq p+1$ by the assumption, so it is valid to assign $h_1=h_2=\cdots=h_p=0$, $h_{r} = \cdots = h_{r+k-1}=h_{r+k+1}= \cdots h_{r+p+1}=0$, and $h_{p+1}=h_{r+k}=1$. Under this assignment, in $\psi(M)$, $h_{r+p+1}-G_{r+p}=0$, $F_{r+p+k}=0$, while $H_{r+p+k} =h_{p+1}h_{r+k}=1\neq 0$. Therefore by Lemma \ref{lem: special matrix, 3} we have $P(\det(M) \mapsto 0) > 1/q$. 

If $p=k-1$, we assign $h_1=h_2=\cdots=h_p=h_{r+p+1}=0$, and assign $h_{p+1}=h_{r+p}=1$. We know that $p<r+p<r+p+1$, so for this assignment to be valid we only need to show that $h_{k+1}$ is not assigned to be both $0$ and $1$. Since $p=k-1$, we have $p<k+1\leq r+p+1$ and the equality is achieved if and only if $r=1$. If $r=1$, then since $r\geq p$ we have $p=1$, $k=2$, which shows $\lambda$ is a staircase. Since we assume $\lambda$ is not a staircase, we must have $r\geq 2$, which means $h_{k+1}$ is not assigned two different values, and the assignment is valid. By a similar argument as before we can show that in $\psi(M)$, $G_{r+p+k-1}=0$, while $H_{r+p+k-1}\neq 0$, showing that $P(\det(M) \mapsto 0) > 1/q$. Combining two cases gives us the desired result.

\end{proof}

\section{Independence Results for families with $P = 1/q$}\label{sec:independence}

In Section \ref{hooks_section}, we saw that it is very easy to find infinite families of hooks where the events of their corresponding Schur funtions being sent to zero are independent. In this section, we investigate the independence of vanishing of the other two shapes with $P(s_\lambda\mapsto 0)=1/q$: rectangles and staircases.

\subsection{Rectangles}

\begin{prop} \label{rectangle_indep}
Fix some $c \in \Z$. Then the events $\{s_{k^{\ell}} \mapsto 0 \mid k-\ell = c\}$ are set-wise independent.
\end{prop}

\begin{proof}
We show the result in the case $c = 0$; the proof is similar for other values of $c$.

We first reduce to showing the following result: Let $k \in \N$ be
arbitrary and $C$ a collection of conditions $\{C_i\}_{i = 1,\ldots,k-1}$,
where $C_i$ is either $s_{i^i} \mapsto 0$ or $s_{i^i} \not\mapsto 0$ for each
$i$. Then 
  \[
    P(s_{k^k} \mapsto 0 \mid C) = 1/q.
  \]
Assuming this result, we show that $P(s_{a^a} \mapsto 0 \mid s_{b^b} \mapsto 0) = 1/q$ for $a \neq b$. 

First, define $\alpha(C)$ to be the number of $i$ such that $C_i$ is $s_{i^i}
\not\mapsto 0$. We have 
  \[
    P(s_{a^a} \mapsto 0 \mid s_{b^b} \mapsto 0) = \frac{P(s_{a^a} \mapsto 0
      \;\&\; s_{b^b} \mapsto 0)}{P(s_{b^b} \mapsto 0)} = q \cdot P(s_{a^a}
    \mapsto 0 \;\&\; s_{b^b} \mapsto 0),
  \]
so it suffices to show that $P(s_{a^a} \mapsto 0 \;\&\; s_{b^b} \mapsto 0) =
1/q^2$. For simplicity, we may assume without loss of generality that $b <
a$. 

We have 
  \begin{align*}
    P(&s_{a^a} \mapsto 0 \;\&\; s_{b^b} \mapsto 0) \\ 
    &= P(s_{a^a} \mapsto 0 \;\&\;
    s_{b^b} \mapsto 0 \;\&\; s_{(a-1)^{a-1}}\mapsto 0) + P(s_{a^a} \mapsto 0
    \;\&\; s_{b^b} \mapsto 0 \;\&\; s_{(a-1)^{a-1}} \not\mapsto 0) \\
    &= \cdots \\
    &= \sum_{C \in \mathcal{C}}{P(s_{a^a} \mapsto 0 \;\&\; s_{b^b} \mapsto 0
      \;\&\; C)},
  \end{align*}
%\Shuli{I am a bit confused here, how many summands do we have here? It seems like we split into two cases each time, so do we have $2^{a-2}$ summmands then? But based on arguments below it seems that we have $q^{a-2}$ summands. I guess my question is what exactly is $C$ and $\mathcal{C}$? If $C$ is a collection of conditions $\{C_i\}_{i = 1,\ldots,k-1}$, where $C_i$ is either $s_{i^i} \mapsto 0$ or $s_{i^i} \not\mapsto 0$ for each $i$ as above, then I don't think the summands will have the same probability, since $s_{i^i} \mapsto 0$ and $s_{i^i} \not\mapsto 0$ don't have the same probability.}
  %\Ben{Very good point. I was sloppy and blatantly incorrect here. I have fixed the probabilities with their correct values. Please see if this makes sense.} 
%\Shuli{It makes sense now. Thanks!}
where $\mathcal{C}$ is the collection of conditions
$C=\{C_i\}_{\substack{1
    \leq i \leq a \\ i \neq a,b}}$. %$\bigcup_{\substack{1\leq i \leq a \\ i \neq a,b}}{C_i}$ 
%\Shuli{I don't quite understand this notation here. What does it mean to union the conditions?}. 
Therefore we reduce to showing that each summand is $(q-1)^{\alpha(C)}/q^a$. 

Indeed, we can write 
  \[
    P(s_{a^a} \mapsto 0 \;\&\; s_{b^b} \mapsto 0 \;\&\; C) = \frac{P(s_{a^a}
      \mapsto 0 \mid s_{b^b} \mapsto 0 \;\&\; C)}{P(s_{b^b} \mapsto 0 \;\&\;
      C)} = \frac{1}{q P(s_{b^b} \mapsto 0 \;\&\; C)}
  \]
by our assumed result. Therefore, induction on $a$ gives that $P(s_{b^b} \mapsto
0 \;\&\; C) = (q-1)^{\alpha(C)}/q^{a-1}$, which gives our desired result.

We are now left with showing the result claimed at the beginning of this proof. We claim this follows immediately from Corollary \ref{cor: geq 1/q}. Indeed, we write
  \[
    P(s_{k^k} \mapsto 0) = P(s_{1^1} \mapsto 0) P(s_{k^k} \mapsto 0 \mid s_{1^1} \mapsto 0) + (1-P(s_{1^1} \mapsto 0)) P(s_{k^k} \mapsto 0 \mid s_{1^1} \not\mapsto 0).
  \]
Expanding each term similarly, we see inductively that we can write
  \[
    P(s_{k^k} \mapsto 0) = \sum_{C = \{C_i\}_{i=1,\ldots k-1}}{a_CP(s_{k^k}
      \mapsto 0 \mid C)}, 
  \]
where $\sum_{C}{a_C} = 1$. Because $P(s_{k^k} \mapsto 0)=1/q$ and each summand
probability is at least $1/q$ as shown in the proof of Corollary \ref{cor: geq 1/q},
each summand must be precisely $1/q$. This
completes the proof. 
\end{proof}

Using the same argument on the collection $\{s_{k^{\ell}} \mid k+\ell = c\}$ gives the following.

\begin{prop}
Fix some $c \in \N$. Then the events $\{s_{k^{\ell}} \mapsto 0 \mid k+\ell = c\}$ are set-wise independent.
\end{prop}

\begin{rem} It is not hard to find examples of rectangles $\lambda$ and $\mu$ such that $P(s_\lambda\mapsto 0 \mid s_\mu\mapsto 0)\neq P(s_\lambda \mapsto 0)$. For example, one can take $\lambda=(2,2)$ and $\mu=(3,3)$. It is straightforward to show that $P(s_{\lambda} \mapsto 0 \;\&\; s_{\mu} \mapsto 0) \neq 1/q^2$ by examining the Jacobi-Trudi matrices, which implies the desired inequality.
\end{rem}

\iffalse
\begin{ex}
Let $\lambda = 2^2$ and $\mu = 3^2$. We show that $P(s_{\lambda} \mapsto 0 \mid s_{\mu} \mapsto 0) \neq 1/q$. For this, it suffices to show that $P(s_{\lambda} \mapsto 0 \;\&\; s_{\mu} \mapsto 0) \neq 1/q^2$.

We have
  \begin{align*}
    s_{\lambda} &= \left\lvert\begin{array}{cc} h_2 & h_3 \\ h_1 & h_2 \end{array}\right\rvert \\
    s_{\mu} &= \left\lvert\begin{array}{cc} h_3 & h_4 \\ h_2 & h_3 \end{array}\right\rvert \\
  \end{align*}
We do casework on whether $h_2 = 0$ and/or $h_1 = 0$. 

\underline{Case 1}: $h_2 = h_1 = 0$. Then we must have $h_3 = 0$. This gives $q$ cases. 

\underline{Case 2}: $h_2 = 0$ and $h_1 \neq 0$. Then we must have $h_3 = 0$, giving $q(q-1)$ possibilities. 

\underline{Case 3}: $h_2 \neq 0$. Then $h_3 = \frac{h_2^2}{h_1}$ and $h_4 = \frac{s_2^3}{s_1^2}$, giving $q(q-1)$ cases. 

\noindent Altogether, we have $2q^2-q$ instances out of $q^4$ where $s_{\lambda} \mapsto 0$ and $s_{\mu} \mapsto 0$. Hence the two events are not independent. 
\end{ex}
\fi

\subsection{Staircases}

We now turn our attention to staircases. We denote staircase $(k,k-1,k-2,\ldots,2,1)$ by $\delta_k$.

\begin{prop} \label{thm: staircase_indep}
For any $k \in \N$ with $k \geq 3$, we have
  \[
  	P(s_{\delta_k} \mapsto 0 \mid s_{\delta_{k-2}} \mapsto 0) = 1/q.
  \]
\end{prop}

\begin{proof}
By Jacobi-Trudi, we have
  \[
    s_{\delta_k} = \begin{vmatrix}
    h_k    & \cdots & h_{2k-4} &h_{2k-3} & h_{2k-2} & h_{2k-1} \\
    \vdots &        &          & \ddots   &          & \vdots \\
    0      & \cdots & h_2      & h_3      & h_4      & h_5 \\
    0      & \cdots & 1        & h_1      & h_2      & h_3 \\
    0      & \cdots & 0        & 0        & 1        & h_1 \\
    \end{vmatrix}.
  \]
Note that $s_{\delta_{k-2}}$ can be obtained from this matrix by removing the first and last rows, along with the last two columns. Hence expanding the determinant about the first row and then last row, we see the determinant will contain a term of the form 
  \[
    h_{2k-1}s_{\delta_{k-2}},
  \]
from the $(1,n)$- and then $(n-1,n-1)$-cofactor. Further, this is the only term in which $h_{2k-1}$ appears. So
  \[
    P(s_{\delta_k} \mapsto 0 \mid s_{\delta_{k-2}} \not\mapsto 0) = 1/q.
  \]
From Theorem \ref{thm: staircase 1/q}, we have
  \begin{align*}
    \frac{1}{q} &= P(s_{\delta_k} \mapsto 0) \\
    &= P(s_{\delta_{k-2}} \mapsto 0)P(s_{\delta_{k}} \mapsto 0 \mid s_{\delta_{k-2}} \mapsto 0) + P(s_{\delta_{k-2}} \not\mapsto 0)P(s_{\delta_{k}} \mapsto 0 \mid s_{\delta_{k-2}} \not\mapsto 0) \\
    &= \frac{1}{q} \cdot \frac{1}{q} + \frac{q-1}{q}P(s_{\delta_{k}} \mapsto 0 \mid s_{\delta_{k-2}} \mapsto 0).
  \end{align*}
Solving this, we obtain the desired implication.
\end{proof}

\begin{cor} \label{cor: staircase_indep1}
For any $k \in \N$, the event $s_{\delta_k} \mapsto 0$ is independent of $s_{\delta_{k+2}} \mapsto 0$. 
\end{cor}

\begin{proof}
We have
  \[
    P(s_{\delta_{k}} \mapsto 0 \;\&\; s_{\delta_{k-2}} \mapsto 0) = P(s_{\delta_{k-2}} \mapsto 0)P(s_{\delta_{k}} \mapsto 0 \mid s_{\delta_{k-2}} \mapsto 0) = \frac{1}{q^2} 
  \]
by Theorem \ref{thm: staircase_indep} and Theorem \ref{thm: staircase 1/q}. Swapping the roles of $k$ and $k-2$ above and again using Theorem \ref{thm: staircase 1/q}, we obtain the desired result.
\end{proof}

\begin{cor} \label{cor: staircase_indep2}
For any $k \in \N$, the event $s_{\delta_k} \mapsto 0$ is independent of $s_{\delta_1} = s_1 = h_1 = e_1 \mapsto 0$.
\end{cor}

\begin{proof}
The proof of Theorem \ref{thm: staircase 1/q} showed that $P(s_{\delta_k} \mapsto 0 \mid h_1 \mapsto 0) = P(s_{\delta_k} \mapsto 0 \mid h_1\mapsto a) = 1/q$, for any $a\in\F_q$. Hence $s_{\delta_k} \mapsto 0$ is independent of $h_1 \mapsto 0$. The reverse follows along the same lines as Corollary \ref{cor: staircase_indep1}.
\end{proof}

\begin{conjecture}
The families $\{s_{\delta_k} \mid k \text{ odd}\}$ and $\{s_{\delta_k} \mid k \text{ even}\}$ are set-wise independent.
\end{conjecture}

\section{Distribution of Other Values}\label{sec:othervalue}
We now broaden our focus to consider not only $P(s_\lambda \mapsto 0)$, but $P(s_\lambda \mapsto a)$ for any $a\in\F_q$, particularly in the case where $\lambda$ is a rectangle.
\subsection{General Results}
\begin{prop}
\label{prop:x^n}
Let $\lambda$ be a partition with $|\lambda| = n$. Then $P(s_\lambda \mapsto a) = P(s_\lambda \mapsto x^na)$ for any $a,x \in \F_q$ with $x \neq 0.$
\end{prop}
\begin{proof}
Since $s_\lambda$ is a homogeneous polynomial of degree $n$ and each $e_i$ is a homogeneous polynomial of degree $i$, if some assignment $e_1 = a_1, \ldots , e_n = a_n$ gives $s_\lambda = a$, then the assignment $e_1 = xa_1, e_2 = x^2a_2,\ldots, e_n = x^na_n$ gives $s_\lambda = x^n a$. Since $x$ is nonzero, this creates a bijection between assignments of the $e_i$ such that $s_\lambda = a$ and assignments of the $e_i$ such that $s_\lambda = x^na$.
\end{proof}

The next result follows directly from Proposition~\ref{prop:x^n}.
\begin{cor}
Let $\lambda$ be a partition with $|\lambda| = n$, and let $q$ be a prime power with $\gcd(n,q-1) = 1$. Then $P(s_\lambda \mapsto a) = P(s_\lambda \mapsto b)$ for any nonzero $a,b \in \F_q$.
\end{cor}

\subsection{Rectangles}
\begin{lem}\label{lem:nonzero}
Let b be a nonzero element of $\F_q$ and let $a \geq n$ be integers, then

\[
P(s_{a^n} \mapsto b) = \sum_{k\geq 1}\sum_{(c_1,c_2,\ldots,c_k) \in C(n)} \frac{(q-1)^{k-1}}{q^n}g_{b}(\gcd(c_1,c_2,\ldots,c_k,q-1)),
\]
where $C(n)$ is the set of all compositions (i.e. ordered partitions) of n, and 

\[
g_b(d) = \begin{cases} 0 & d \nmid \frac{q-1}{ord(b)} \\ d & d | \frac{q-1}{ord(b)} \end{cases},
\]
where $ord(b)$ is the order of $b$ in the multiplicative group $\F_q^{\times}$.
\end{lem}

Before proving Lemma~\ref{lem:nonzero}, we require the following definitions and lemmas.

Recall the definition of the map $\varphi$, based on $\psi$ from Section \ref{sec:operations}, which takes a general Schur matrix and a set of assignments to a reduced general Schur matrix. For the following proofs we use a modified version of both $\varphi$ and $\psi$, which we will denote $\tilde{\varphi}$ and $\tilde{\psi}$. While $\varphi$ and $\psi$ reduce the size of the matrix, $\tilde{\varphi}$ and $\tilde{\psi}$ will leave the size unchanged.

\begin{defn}
Let $M$ be a general Schur matrix of size $n$ with $m$ variables. We define an operation $\tilde\psi$ that takes general Schur matrices to matrices over $\F_q[x_1,x_2,\ldots,x_m]$ as follows.
\begin{enumerate}[(a)]
\item If $M$ has no nonzero constant entries, then $\psi(M) = M$.
\item If $M$ has $k \geq 1$ nonzero constant entries:
\begin{enumerate}
\item[(i)] From top to bottom, use each of these $k$ entries as a pivot to make all other entries in its column zero by subtracting appropriate multiples of its row from each of the rows above.
\item[(ii)] Then, use this nonzero constant to make all other entries in its row zero using column operations.
\end{enumerate}
The resulting matrix $M'$ is $\tilde\psi(M)$.
\end{enumerate}

\iffalse
Let $M$ be a general Schur matrix with m free variables. Define an operation $\tilde{\psi}$ that takes general Schur matrices with m free variables to matrices over $\F_q[x_1,x_2,\ldots,x_k]$: 
\begin{enumerate}[(a)]
\item If $M$ has no nonzero constants as entries, then $\psi(M) = M$.
\item If $M$ has $k \geq 1$ many nonzero constant entries, then from top to bottom, for each of these $k$ entries we use it as a pivot to turn all the other entries in its column into zero by subtracting multiple of the row it is in from each of the rows above. Then we further use these nonzero constants to turn all the other entries in the their rows into zero by column operations, giving a new matrix $M'$. Unlike $\psi$, we do not delete these rows and columns. Define $\tilde{\psi}(M) = M'$ in this case.
\end{enumerate}
\fi
\end{defn}

Notice that $\tilde{\psi}$ only performs determinant-preserving row and column operations on the matrix, so $\det(M) = \det{\tilde{\psi}(M)}$. To see an example of $\tilde{\psi}$, look at the first step of Example~\ref{ex:psi}.

We define $\tilde{\varphi}$ analogously to $\varphi$ by replacing $\psi$ in the definition of $\varphi$ with $\tilde{\psi}$. For the following lemma, we rename the $h_i$ by $x_i$.

\begin{lem}
Let $\lambda = (a^n)$ be a rectangular partition, and let $A = (x_{j-i+n})_{1\leq i,j\leq n}$ be the Jacobi-Trudi matrix corresponding to $s_\lambda$. Then for any $1\leq r\leq m$ and any choice of $a_1,\ldots,a_r\in \F_q$:
\begin{enumerate}
\item$\tilde{\varphi}(A; x_1=a_1,x_2=a_2,\ldots,x_r=a_r)$ is a block anti-diagonal matrix.
\item The top right block is $\varphi(A; x_1=a_1,x_2=a_2,\ldots,x_r=a_r)$.
\item All others blocks are either scalar multiples of the identity or the zero matrix.
\end{enumerate}
\end{lem}

\begin{ex}\label{ex:tilpsi}
Let $A$ be the Jacobi-Trudi matrix corresponding to $\lambda=(4^4)$. We show below the matrix $A$ as well as the result of an application of $\varphi$ and $\tilde\varphi$. 
\begin{align*}
A=&\begin{bmatrix}
x_4 & x_5 & x_6 & x_7 \\
x_3 & x_4 & x_5 & x_6 \\
x_2 & x_3 & x_4 & x_5 \\
x_1 & x_2 & x_3 & x_4
\end{bmatrix}\\
\tilde{\varphi}(A; x_1 = 0, x_2 = 2) = &
\begin{bmatrix}
0 & 0 & x_6 - \frac{x_3x_5 - x_4^2}{2} & x_7-x_4x_5 \\
0 & 0 & x_5-x_3x_4 & x_6- \frac{x_3x_5 - x_4^2}{2} \\
2 & 0 & 0 & 0 \\
0 & 2 & 0 & 0
\end{bmatrix}\\
\varphi(A; x_1 = 0, x_2 = 2) = &
\begin{bmatrix}
x_6 - \frac{x_3x_5 - x_4^2}{2} & x_7-x_4x_5 \\
x_5-x_3x_4 & x_6- \frac{x_3x_5 - x_4^2}{2} \\
\end{bmatrix}
\end{align*}
\end{ex}

\begin{proof}
We proceed by induction on $r$. If $r=1$, then $\tilde{\varphi}(A; x_1 = a_1)$ is either equal to  $\varphi(A; x_1 = a_1)$ or can be decomposed into a $1\times1$ nonzero block and a block equal to $\varphi(A; x_1 = a_1)$.

For the inductive step, assume $\tilde{\varphi}(A; x_1 = a_1, x_2 = a_2, \ldots , x_{r-1} = a_{r-1})$ is of the desired form. Then $\tilde{\varphi}(A;x_1 = a_1, \ldots , x_r = a_r) = \tilde{\varphi}(\tilde{\varphi}(A;x_1 = a_1, \ldots x_{r-1} = a_{r-1} ); x_r = a_r) $ will only change the final block of $\tilde{\varphi}(A; x_1 = a_1, x_2 = a_2, \ldots , x_{r-1} = a_{r-1})$ that is equal to $\varphi(A; x_1 = a_1, x_2 = a_2, \ldots , x_{r-1} = a_{r-1})$, since it is the only block in which new nonzero constants can appear. Thus it suffices to show that $\tilde{\varphi}(\varphi(A; x_1 = a_1, x_2 = a_2, \ldots , x_{r-1} = a_{r-1});  x_r = a_r)$ is of the desired form.

By \ref{lem:rec}, the entries on the lowest nonzero diagonal of $\varphi(A; x_1 = a_1, x_2 = a_2, \ldots , x_{r-1} = a_{r-1})$ are all the same. The new assignment $x_r = a_r$ can set only these entries to some nonzero constant, since only they have the smallest label. If $x_r = a_r$ does set these entries to some nonzero constant and these entries are below the main diagonal, then $\tilde{\varphi}(\varphi(A;x_1=a_1,\ldots,x_{r-1} = a_{r-1});x_r = a_r)$ is block-antidiagonal with two blocks, where the lower left block is a scalar multiple of the identity, and the upper right block is $\varphi(A;x_1=a_1,\ldots,x_r = a_r)$. If these entries are set to some nonzero constant and are above the main diagonal, then $\tilde{\varphi}(\varphi(A;x_1=a_1,\ldots,x_{r-1} = a_{r-1});x_r = a_r)$ is block-antidiagonal with two blocks: a block of zeros in the bottom left and a scalar multiple of the identity in the top right. If these entries are not sent to 0, then $\tilde{\varphi}(\varphi(A;x_1=a_1,\ldots,x_{r-1} = a_{r-1});x_r = a_r) = \varphi(A;x_1=a_1,\ldots,x_r = a_r)$. All possibilities are of the desired form.
\end{proof}

If the determinant of $A= (x_{j-i+n})_{1\leq i,j\leq n}$ under some particular assignment $x_1 = a_1, x_2 = a_2,\ldots,x_{2n-1}=a_{2n-1}$ is nonzero, then each block in $\tilde{\varphi}(A; x_1 = a_1,\ldots, x_{2n-1} = a_{2n-1})$ must have nonzero determinant. Thus each block must be a scalar multiple of the identity, since $\varphi(A;x_1=a_1,\ldots,x_{2n-1}=a_{2n-1})$ must either be empty or only contain 0's. The sizes of these blocks therefore must sum to $n$. This observation motivates the following definition.

\begin{defn}
Let $A = (x_{j-i+n})_{1\leq i,j\leq n}$ be a matrix corresponding to some rectangular Schur function, and let $x_1 = a_1, x_2 = a_2,\ldots,x_{2n-1} = a_{2n-1}$ be an assignment of the $x_i$'s such that A is nonsingular. Define the \textit{block structure of A under assignment $(a_1,a_2,\ldots,a_{2n-1})$} to be the sequence $(c_1,...,c_k)$ of the sizes of blocks of $\tilde{\varphi}(A; x_1 = a_1, x_2 = a_2,\ldots, x_{2n-1} = a_{2n-1})$ from the lower left corner to the upper right corner. 
Denote this by \[B(A; x_1 = a_1, x_2 = a_2,\ldots,x_{2n-1} = a_{2n-1}) = (c_1,c_2,\ldots,c_k).\]
\end{defn}

\begin{ex}\label{ex:blockstruct}
Let $A$ be the Jacobi-Trudi matrix corresponding to $\lambda=(3^3)$. Consider the assignment $(x_1 = 0, x_2 = 2, x_3 = 1, x_4 = 1, x_5 = 4)$.
\begin{align*}
A=&\begin{bmatrix}
x_3 & x_4 & x_5 \\
x_2 & x_3 & x_4 \\
x_1 & x_2 & x_3 
\end{bmatrix}\\
\tilde{\varphi}(A; x_1 = 0, x_2 = 2) = &
\begin{bmatrix}
0 & 0 & x_5-x_3x_4 \\
2 & 0 & 0\\
0 & 2 & 0
\end{bmatrix}\\
\tilde{\varphi}(A; x_1 = 0, x_2 = 2, x_3 = 1, x_4 = 1, x_5 = 4) = &
\begin{bmatrix}
0 & 0 &3\\
2 & 0 & 0\\
0 & 2 & 0
\end{bmatrix}\\
B(A; x_1 = 0, x_2 = 2, x_3 = 1, x_4 = 1, x_5 = 4)  = & (2,1)
\end{align*}

\end{ex}

We are now ready to prove Lemma~\ref{lem:nonzero}, which is the main lemma of this section.

\begin{proof}[Proof of Lemma~\ref{lem:nonzero}]
We now consider the probability that for some fixed $A$ and composition $(c_1,\ldots,c_k)$, $B(A; x_1 = a_1,\ldots,x_{2n-1} = a_{2n-1}) = (c_1,\ldots,c_k)$. Within a block of size $\ell$, the $\ell -1$ diagonals below the main diagonal must be set to 0, and the main diagonal must be nonzero. The values of the diagonals above the main diagonal can be anything. Therefore,
\[
P(B(A; x_1 = a_1,\ldots,x_{2n-1} = a_{2n-1}) = (c_1,\ldots,c_k)) = \left(\frac{1}{q}\right)^{n-k}\left(\frac{q-1}{q}\right)^k = \frac{(q-1)^k}{q^n}.
\]

The determinant of a matrix of this form is $\pm y_1^{c_2}y_2^{c_2}\cdots y_k^{c_k}$, where the $y_i$ are the values on the main diagonal of the $i^{th}$ block. From the definition of $\tilde\varphi$ and the structure of $A$ we can see that these $y_i$ are all independent and uniformly distributed across the nonzero values of $\F_q$. Also note that if the determinant is $-y_1^{c_2}y_2^{c_2}\cdots y_k^{c_k}$, then some $c_i$ must be odd, and replacing $y_i$ with $-y_i$ allows us to ignore the sign.

We now want to calculate the probability that $y_1^{c_1}y_2^{c_2}\cdots y_k^{c_k} = b$. If $c_1 = c_2 = \ldots = c_k$, then this is just the probability that for some nonzero $y \in \F_q$, $y^{c_1} = b$, so 
\[
P(y_1^{c_1}y_2^{c_2}\cdots y_k^{c_k} = b) = \frac{g_b(\gcd(q-1,c_1))}{q-1}.
\]

If some $c_i > c_j$, then $$y_1^{c_1}y_2^{c_2}\cdots y_k^{c_k} = y_1^{c_1}y_2^{c_2}\cdots y_i^{c_i-c_j} \cdots y_{j-1}^{c_{j-1}}(y_iy_j)^{c_j}y_{j+1}^{c_{j+1}} \cdots y_k^{c_k}.$$
Since $y_iy_j$ is still uniformly distributed and independent of $y_h$ when $h \neq j$,
\begin{align*}
&P(y_1^{c_1}y_2^{c_2}\cdots y_i^{c_i-c_j} \cdots y_{j-1}^{c_{j-1}}(y_iy_j)^{c_j}y_{j+1}^{c_{j+1}} \cdots y_k^{c_k} = b) \\
=&P(y_1^{c_1}y_2^{c_2}\cdots y_i^{c_i-c_j} \cdots y_{j-1}^{c_{j-1}}y_j^{c_j}y_{j+1}^{c_{j+1}} \cdots y_k^{c_k} = b).
\end{align*}
We can then use the Euclidean algorithm to obtain
\[
P(y_1^{c_1}y_2^{c_2}\cdots y_k^{c_k} = b) =  P(y_1^dy_2^d\cdots y_k^d = b) = \frac{g_b(\gcd(q-1,d))}{q-1},
\]
where $d = \gcd(c_1,c_2,\ldots,c_k)$.

Therefore we have 
\begin{align*}
&P\big(s_{a^n} \mapsto b\ \ \&\ \ B(A; x_1 = a_1,\ldots,x_{2n-1} = a_{2n-1}) = (c_1,\ldots,c_k)\big) \\
=& \frac{(q-1)^{k-1}}{q^n}g_b(\gcd(q-1,c_1,\ldots,c_k)).
\end{align*}
Summing over all compositions gives the lemma.
\end{proof}

We next remind the reader of the definition of the M\"obius function of integers, $\mu$. Let $n$ be an integer. Then $\mu(n)=0$ if $n$ is not square-free, and $\mu(n)=(-1)^p$, where $p$ is the number of prime factors of $n$, otherwise. For example, $\mu(4)=0$ and $\mu(30)=(-1)^3$.

Lemma~\ref{lem:nonzero} allows us to write $P(s_{a^n}\mapsto b)$ explicitly, as shown in the following theorem.

\begin{prop}
Let $b$ be a nonzero element of $\F_q$, and let $a \geq n$. Then
\[
P(s_{a^n} \mapsto b) = \sum_{d|\gcd(q-1,n)} \frac{f_b(d)}{q^{n(d-1)/d +1}},
\]
where 
\[
	f_b(d) = \sum_{e|d}\mu(e)g_b\left(\frac{d}{e}\right)
\]
is the M{\"o}bius inverse of $g_b$.
\end{prop} 

\begin{proof}
By Lemma~\ref{lem:nonzero}, we have 
\[
P(s_{a^n} \mapsto b) =\sum_{k\geq 1} \sum_{(c_1,c_2,\ldots,c_k) \in C(n)} \frac{(q-1)^{k-1}}{q^n}g_{b}(\gcd(c_1,c_2,\ldots,c_k,q-1)).
\]
Note that the summand only depends on $\gcd(c_1,c_2,\ldots,c_k,q-1)$ and $k$, so we can rewrite the sum as
\[
P(s_{a^n} \mapsto b) = \sum_{\substack{d|\gcd(n,q-1)\\1\leq k \leq n/d}} N(d,k) \frac{(q-1)^{k-1}}{q^n}f_b(d),
\]
where $N(d,k)$ counts the number of compositions $(c_1,\ldots,c_k)$ of $n$ with $k$ parts such that $d$ divides $\gcd(c_1,c_2,\ldots,c_k,q-1)$ and
\[
\sum_{e|d} f_b(e) = g_b(d),
\]
or equivalently, by M{\"o}bius inversion, 
\[
	f_b(d) = \sum_{e|d}\mu(e)g_b(\frac{d}{e}).
\]
This switch from $g_b$ to $f_b$ is necessary because some compositions are counted by multiple $N(d,k)$.

It is easy to compute $N(d,k)$ as we can consider a $k$-part composition of $n$ as a choice of $k-1$ break points along a line of length $n$. To satisfy the gcd requirement we only have $\frac{n}{d} -1$ places where these break points can be, thus $N(d,k) = \binom{n/d-1}{k-1}$.

We can then simplify the sum using the binomial formula:
\[
P(s_{a^n} \mapsto b) = \sum_{\substack{d|\gcd(n,q-1)\\1\leq k \leq n/d}} \binom{n/d-1}{k-1} \frac{(q-1)^{k-1}}{q^n}f_b(d) = \sum_{d|\gcd(n,q-1)}  \frac{q^{n/d-1}}{q^n}f_b(d).
\]
This completes the proof.
\end{proof}

Recall that Eurler's totient function $\phi$ applied to a positive integer $m$ counts the number of positive integers up to $m$ that are relatively prime to $m$.

\begin{cor}
Let $b$ be an element of $\F_q$ of multiplicative order $q-1$, and let $a \geq n$. 
Then 
\[
P(s_{a^n} \mapsto b) = \sum_{d|\gcd(q-1,n)} \frac{\mu(d)}{q^{n(d-1)/d +1}},
\]
and
\[
P(s_{a^n} \mapsto 1) = \sum_{d|\gcd(q-1,n)} \frac{\phi(d)}{q^{n(d-1)/d +1}},
\]
where $\mu$ is the M{\" o}bius function and $\phi$ is Euler's totient function.
\end{cor}
\begin{proof}
These are both immediate consequences of the preceding theorem, when $g_b$ is particularly nice, and therefore $f_b$ is easy to compute.
\end{proof}

\section{Results and conjectures for miscellaneous shapes}\label{sec:misc}
We have a handful of miscellaneous results that we will list for the interested reader. We omit the proofs for brevity and instead simply discuss the results.
\subsection{Probability $\frac{q^2+q-1}{q^3}$}
We first identify some shapes $\lambda$ such that $P(s_\lambda\mapsto 0 )=\frac{q^2+q-1}{q^3}$:
\begin{itemize}
\item $\lambda = (a,a-1,a-2)$ with $a \geq 5$,
\item $\lambda = (a,b,1^{m})$ with $b \geq 2$ and $a \neq b + m$, and
\item $\lambda=(a^m,1^n)$ where $a,m>1, n \geq 1$.
\end{itemize}

Note that we cannot drop the condition $a \neq b+m$ in the second family of shapes. For example, if $\lambda=(4,3,1)$, then $P(s_\lambda\mapsto 0)=(q^3+q^2-2q+1)/q^4$. Regarding the third family, note that not all fattened hooks have the same probability. In fact, it seems the probability for a fattened hook can be very complicated, and Proposition \ref{prop: quasi-polynomial} gives an example.

\subsection{Generalized staircases} We also work on generalization of staircases. Based on numerical data for $n$ up to $7$, we have the following conjecture.
\begin{conjecture}
For a 2-staircase $\lambda=(2n, 2n-2, \cdots, 4, 2)$, we have $$P(s_\lambda\mapsto0)={(q^2+q-1)}/{q^3}.$$
\end{conjecture}

\subsection{Relaxing Proposition~\ref{prop: general shape}}Proposition \ref{prop: general shape} computes the probability for shapes where all the row sizes are far apart, in which case we have no constants or repeated variable in the Jacobi-Trudi matrix. If we relax the condition so that we have one repeated variable, then we obtain the following proposition.

\begin{prop}
\label{prop:relaxed strictness}
Let $\lambda = (\lambda_1,\ldots,\lambda_k)$, where $\lambda_j - \lambda_{j + 1} = k - 2$ for some $j < k$, $\lambda_i - \lambda_{i + 1} \geq k - 1$ for all $i < k, i \neq j$ and $\lambda_k \geq k$. Then
$$P(s_{\lambda} \mapsto 0) = \frac{1}{q^{k^2-2k+2}} \left( q^{k^2-2k+2} - (q^{2k-2}-q^{k-1}-q^{k-2}+1) \prod_{i=0}^{k-3}(q^{k-2} - q^i) \right).$$
\end{prop}

\noindent\textbf{Acknowledgments.}
This research was carried out as part of the 2016 REU program at the School of Mathematics at University of Minnesota, Twin Cities, and was supported by NSF RTG grant DMS-1148634 and by NSF grant DMS-1351590.  The authors would like to thank Joel Lewis and Ben Strasser for their comments and suggestions.  The authors are also grateful to Ethan Alwaise who helped get this project started. 

\bibliographystyle{alpha}
\bibliography{main}
\end{document}